\documentclass[11pt]{article}
\usepackage[top=27mm, bottom=27mm, left=22mm, right=22mm]{geometry}
\usepackage[square, comma, sort&compress, numbers]{natbib}
\usepackage{amsthm}
\usepackage{hyperref}
\usepackage{amsfonts}
\usepackage{mathrsfs}
\usepackage{comment}
\usepackage{amsmath}
\usepackage{amssymb}
\usepackage{amscd}
\usepackage{graphicx}
\usepackage{booktabs,tabularx,array}
\usepackage{booktabs,tabularx,array,multirow,amsmath}
\usepackage[all]{xy}
\usepackage{titlesec}
\usepackage{enumerate}
\usepackage{bm}
\usepackage{relsize}
\usepackage{enumitem}
\usepackage{color}
\usepackage{dsfont}
\usepackage{arydshln}
\usepackage[capitalize]{cleveref}
\usepackage{appendix}
\usepackage{ulem}
\newtheorem{theorem}{Theorem}[section]
\newtheorem{lemma}[theorem]{Lemma}
\newtheorem{proposition}[theorem]{Proposition}

\newtheorem{conjecture}[theorem]{Conjecture}
\newtheorem{definition}[theorem]{Definition}

\newtheorem{claim}[theorem]{Claim}

\newtheorem{fact}[theorem]{Fact}

\newcommand{\ma}{\mathcal}

\newcommand{\qbinom}[2]{\genfrac{[}{]}{0pt}{}{#1}{#2}}
\global\long\def\f{\mathcal{F}}

\begin{document}

\title{An Erd\H{o}s Matching Conjecture for Vector Spaces}

\date{}

\makeatletter
\def\thanks#1{\protected@xdef\@thanks{\@thanks
        \protect\footnotetext{#1}}}
\makeatother

\author{Baoyan Feng, Chong Shangguan, Yulin Yang, Chenyang Zhang\thanks{The authors are with the Research Center for Mathematics and Interdisciplinary Sciences, Shandong University, Qingdao 266237, China, and the Frontiers Science Center for Nonlinear Expectations, Ministry of Education, Qingdao 266237, China. Emails: (fengbaoyan2002@163.com, theoreming@163.com, forestyoung@mail.sdu.edu.cn, lxzcyang@163.com)}}
\maketitle

\begin{abstract}
We study a vector-space analogue of the Erd\H{o}s Matching Conjecture.
Let $m_q(n,k,s)$ denote the maximum cardinality of a family of $k$-dimensional subspaces of an $n$-dimensional vector space over $\mathbb F_q$ with no $s+1$ members whose sum is direct. 
Two natural constructions provide lower bounds. The first consists of all $k$-subspaces contained in a fixed $((s+1)k-1)$-dimensional subspace; 
the second consists of all $k$-subspaces that intersect a fixed $s$-dimensional subspace nontrivially. These constructions motivate the following vector-space analogue of the Erd\H{o}s Matching Conjecture:
for all $n\ge (s+1)k$,
$$m_q(n,k,s)=\max\left\{\qbinom{(s+1)k-1}{k}_q,~\qbinom{n}{k}_q-q^{ks}\qbinom{n-s}{k}_q\right\}.$$

We prove this conjecture when $k=2$, when $n=(s+1)k$, and when $n$ is sufficiently large. 
In particular, the case $k=2$ may be viewed as a vector-space analogue of the Erd\H{o}s--Gallai theorem. 
In the large-$n$ range, we also prove a Hilton--Milner-type stability theorem, determining the largest nontrivial families with this property. 
Finally, we connect this problem with $t$-cover-free families in vector spaces and determine their extremal number up to a lower-order term, extending a recent result of Shan and Zhou for the special case $t=2$. 
The proofs combine Lov\'asz's minimax theorem for matroid matchings, a high-dimensional Hoffman bound for uniform hypergraphs, and packing-design arguments in vector spaces. 
\end{abstract}


\section{Introduction}

For integers $n\ge k\ge 1$, write $[n]:=\{1,\ldots,n\}$ and let $\binom{[n]}{k}$ denote the family of all $k$-subsets of $[n]$. A $k$-uniform hypergraph on $n$ vertices is a family $\mathcal{F}\subseteq \binom{[n]}{k}$, whose members are called edges. The {\it matching number} $\nu(\mathcal{F})$ is the maximum size of a collection of pairwise disjoint edges in $\mathcal{F}$. A central problem in extremal combinatorics is to determine the maximum number of edges in a $k$-uniform hypergraph with bounded matching number. For $s\ge 1$, let $m(n,k,s)$ denote the maximum size of a family $\mathcal{F}\subseteq \binom{[n]}{k}$ satisfying $\nu(\mathcal{F})\le s$; equivalently, $\mathcal{F}$ contains no $s+1$ pairwise disjoint edges. When $n\le (s+1)k-1$, it is immediate that $m(n,k,s)=\binom{n}{k}$. For $n\ge (s+1)k$, there are two natural constructions that provide lower bounds for $m(n,k,s)$:
\begin{align*}
    \mathcal{A}(k,s)=\binom{[(s+1)k-1]}{k},~
    \mathcal{B}(n,k,s)=\left\{F\in\binom{[n]}{k}:F\cap [s]\neq\varnothing\right\}.
\end{align*}

In 1965, Erd\H{o}s \cite{erdos1965problem} famously conjectured that one of these two constructions is optimal.

\begin{conjecture}[Erd\H{o}s Matching Conjecture (EMC for short) \cite{erdos1965problem}]\label{conj:emc}
For all $n\ge (s+1)k$,
\begin{equation*}
m(n,k,s)=\max\left\{\binom{(s+1)k-1}{k},~\binom{n}{k}-\binom{n-s}{k}\right\}.
\end{equation*}
\end{conjecture}

The EMC has been proved for several special cases. When $s=1$, it reduces to the classical Erd\H{o}s--Ko--Rado theorem \cite{erdos1961intersection}. The case $k=1$ is immediate, while the case $k=2$ was settled by Erd\H{o}s and Gallai \cite{erdHos1959maximal}; see also \cite{akiyama1985size,ning2020formula} for two different proofs. The case $k=3$ was resolved through a sequence of works: Frankl, R\"odl, and Ruci\'nski \cite{frankl2012maximum} proved the conjecture for $n\ge 4s$, Łuczak and Mieczkowska \cite{luczak2014erdHos} established it for all $n$ and sufficiently large $s$, and Frankl \cite{frankl2017maximum} completed the proof for $k=3$. Very recently, Frankl, Lu, Ma, and Wu proved the conjecture for $k=4$ and sufficiently large $n\ge 5s$ \cite{frankl2026towards}, and Hou, Hu, and Liu \cite{hou2026finite} proved the conjecture for $s\ge 6961$.

For general $s$ and $k$, a number of partial results are known. The case $n=(s+1)k$ was already implicit in the work of Kleitman \cite{kleitman1968maximal}. In his original paper, Erd\H{o}s \cite{erdos1965problem} proved the conjecture for all $n\ge n_0(k,s)$. Bollob\'as, Daykin, and Erd\H{o}s \cite{bollob1976sets} later improved this to $n\ge 2k^3s$. They also obtained a Hilton--Milner-type stability theorem. After more than thirty years of limited progress, Huang, Loh, and Sudakov \cite{huang2012size} showed that the conjecture holds for $n\ge 3k^2s$. This was further improved by Frankl, Luczak, and Mieczkowska \cite{frankl2012matchings}, who proved the conjecture for $n\ge 2(s+1)\frac{k^2}{\log k}$, and later by Frankl \cite{frankl2013improved}, who established it for $n\ge (2s+1)k-s$. More recently, Frankl and Kupavskii \cite{frankl2022erdHos} proved that the EMC holds whenever $n\ge \frac{5}{3}sk-\frac{2}{3}s$ for sufficiently large $s$. On the other hand, for small $n$, Frankl \cite{frankl2017proof} proved the conjecture for $n\le (s+1)(k+\varepsilon)$, where $\varepsilon$ depends on $k$, and Kolupaev and Kupavskii \cite{kolupaev2023erdHos} further proved it for $n\le (s+1)(k+\frac{1}{100k})$, provided that $s>101k^3$.

Many problems in extremal set theory admit natural analogues in the vector-space setting; see \cite{hsieh1975intersection,frankl1986erdos,chowdhury2010shadows,chowdhury2012shadows,frankl1985intersection,ihringer2021remarks,ihringer2026structure,liu2024mathcal,liao2025isodiametric,etzion2013problems} for several examples related to the present work. Such vector-space analogues are often more subtle than their set-system counterparts, since many techniques developed for finite set systems do not readily extend to vector spaces. For instance, the shifting technique is a powerful tool in the study of extremal problems for set systems \cite{frankl1987shifting}. However, as far as we know, despite considerable effort, no satisfactory vector-space analogue of this technique has yet been developed.

In this paper, we formulate a vector-space analogue of Conjecture~\ref{conj:emc} and establish a connection between it and cover-free families in vector spaces.

\subsection{An Erd\H{o}s Matching Conjecture for vector spaces}

Throughout the paper, $q$ denotes a prime power, and $V$ denotes an $n$-dimensional vector space over $\mathbb{F}_q$. Let $\qbinom{V}{k}$ denote the set of all $k$-dimensional subspaces ($k$-subspaces for short) of $V$. We write $\qbinom{n}{k}_q$ for the Gaussian binomial coefficient, namely, the number of $k$-subspaces of $\mathbb{F}_q^n$.

The vector-space Erd\H{o}s--Ko--Rado theorem (see, e.g., \cite{hsieh1975intersection,frankl1986erdos}) asks for the maximum size of a family $\mathcal{F}\subseteq\qbinom{V}{k}$ in which no two members intersect {\it trivially}; that is, $F_1\cap F_2\neq\{0\}$ for all $F_1,F_2\in\mathcal{F}$, where $0$ denotes the zero vector.

There are at least two natural directions in which one can extend this problem by forbidding larger ``matchings'' in $\qbinom{V}{k}$. One direction, considered by Ihringer \cite{ihringer2021remarks} and by Liu, Yu, Feng, and Li \cite{liu2024mathcal}, is to determine the maximum size of a family $\mathcal{F}\subseteq\qbinom{V}{k}$ containing no $s+1$ members that pairwise intersect trivially.

In this paper, we investigate a different generalization. We study the maximum size of a family $\mathcal{F}\subseteq\qbinom{V}{k}$ containing no $s+1$ members whose sum is {\it direct}. This notion is motivated by its connections to several topics in matroid theory, extremal set theory, and coding theory, including the matroid matching problem \cite{lovasz1980selecting,lovasz1980matroid}, $r$-wise $t$-intersecting families of vector spaces \cite{cao2023r,chowdhury2010shadows,chowdhury2012shadows}, cover-free families of vector spaces \cite{shan2026cover}, and covering Grassmannian codes \cite{etzion2019grassmannian,qian2023covering}.

More formally, for a family $\mathcal{F}\subseteq \qbinom{V}{k}$, let $\nu_q(\mathcal{F})$ denote the largest integer $s$ such that there exist $F_1,\dots,F_s\in\mathcal{F}$ whose sum is direct, that is, $\dim(F_1+\cdots+F_s)=sk$. Thus the problem may be viewed as an extremal problem for matchings in a linear-matroidal setting: the members of $\mathcal F$ are $k$-flats, and a matching is a collection whose union has full rank. In analogy with $m(n,k,s)$, let $m_q(n,k,s)$ denote the maximum size of a family $\mathcal{F}\subseteq \qbinom{V}{k}$ satisfying $\nu_q(\mathcal{F})\le s$.

For $n\le (s+1)k-1$, it is clear that $m_q(n,k,s)=\qbinom{n}{k}_q$. For $n\ge (s+1)k$, there are two straightforward vector-space analogues of $\mathcal{A}(k,s)$ and $\mathcal{B}(n,k,s)$, each containing at most $s$ members whose sum is direct:
\begin{align}
    \mathcal{A}_q(k,s)&=\qbinom{V_1}{k},\text{ where }V_1\text{ is an }((s+1)k-1)\text{-subspace},\label{eq:f1}\\
    \mathcal{B}_q(n,k,s)&=\left\{F\in \qbinom{V}{k}:F\cap V_2\neq\{0\}\right\},\text{ where }V_2\text{ is an }s\text{-subspace}.\label{eq:f2}
\end{align}
It is easy to verify that $|\mathcal{A}_q(k,s)|=\qbinom{(s+1)k-1}{k}_q$ and $|\mathcal{B}_q(n,k,s)|=\qbinom{n}{k}_q-q^{sk}\qbinom{n-s}{k}_q$ (see also \eqref{eq:B_q(n,k,s)}).

We now formulate the following vector-space analogue of Conjecture~\ref{conj:emc}, which asserts that the extremal value of $m_q(n,k,s)$ is attained by one of the two constructions above.

\begin{conjecture}[vector-space EMC]
\label{conj:vsemc}
For all $n\ge (s+1)k$,
\begin{equation*}
    m_q(n,k,s)=\max\left\{\qbinom{(s+1)k-1}{k}_q,~\qbinom{n}{k}_q-q^{sk}\qbinom{n-s}{k}_q\right\}.
\end{equation*}
\end{conjecture}

Since $\nu_q(\mathcal{F})\le 1$ if and only if no two members of $\mathcal{F}$ intersect trivially, the vector-space Erd\H{o}s--Ko--Rado theorem \cite{hsieh1975intersection,frankl1986erdos} implies that $m_q(n,k,1)=\qbinom{n-1}{k-1}_q$ for all $n\ge 2k$. It is also immediate that $m_q(n,1,s)=\qbinom{s}{1}_q$.

In this paper, we prove Conjecture~\ref{conj:vsemc} for the following ranges of parameters.

\begin{theorem}\label{thm:vemc}
Let $n$, $k$, and $s$ be positive integers, and let $q$ be a prime power. Let $V$ be an $n$-dimensional vector space over $\mathbb{F}_q$.
\begin{itemize}
    \item[{\rm (i)}] For $k=2$ and $n\ge 2(s+1)$,
    \begin{equation}\label{eq:vector-space-emc-k=2}
        m_q(n,2,s)=\max\left\{\qbinom{2s+1}{2}_q,~\qbinom{n}{2}_q-q^{2s}\qbinom{n-s}{2}_q\right\}.
    \end{equation}
    Moreover, the extremal families are precisely those among $\mathcal A_q(2,s)$ and $\mathcal B_q(n,2,s)$ whose sizes attain the maximum in \eqref{eq:vector-space-emc-k=2}.

    \item[{\rm (ii)}] For all $n\ge (s+1)k$,
    \begin{equation}\label{eq:vector-space-emc-general}
        m_q(n,k,s)\le \qbinom{s}{1}_{q^k}\qbinom{n-1}{k-1}_q.
    \end{equation}
    In particular, when $n=(s+1)k$,
    \begin{equation}\label{eq:vector-space-emc-n=(s+1)k}
        m_q((s+1)k,k,s)=\qbinom{(s+1)k-1}{k}_q.
    \end{equation}
    Moreover, a family $\mathcal{F}\subseteq\qbinom{V}{k}$ is extremal if and only if it is isomorphic to $\mathcal{A}_q(k,s)$ unless $s=1$ and $n = 2k$.
        
    \item[{\rm (iii)}] For $n\ge (2s+1)k-s+1$,
    \begin{equation}\label{eq:vector-space-emc-large-n}
        m_q(n,k,s)=\qbinom{n}{k}_q-q^{sk}\qbinom{n-s}{k}_q.
    \end{equation}
    Moreover, a family $\mathcal{F}\subseteq\qbinom{V}{k}$ is extremal if and only if it is isomorphic to $\mathcal{B}_q(n,k,s)$.
\end{itemize}
\end{theorem}

We have several remarks on \cref{thm:vemc}.
First, \eqref{eq:vector-space-emc-k=2} may be viewed as a vector-space analogue of the Erd\H{o}s--Gallai theorem \cite{erdHos1959maximal}, which settles the case $k=2$ of Conjecture~\ref{conj:emc}. The proof of \eqref{eq:vector-space-emc-k=2} relies on a result of Lov\'asz \cite{lovasz1980selecting} on selecting independent lines from a family of lines in a projective space (see Lemma~\ref{lem: min-max}). In the present setting, these projective lines correspond precisely to the $2$-dimensional subspaces of the ambient vector space. Thus, Lov\'asz's result is a projective-geometric analogue of the Tutte--Berge formula \cite{berge1958couplage} for graph matchings: it provides a structural description of families of lines with bounded $\nu_q(\cdot)$ and reduces the corresponding extremal problem to a counting argument.

Second, in the set-system setting, Frankl \cite{frankl1987shifting} proved that $m(n,k,s)\le s\binom{n-1}{k-1}$. As a vector-space analogue of this bound, our proof of \eqref{eq:vector-space-emc-general} combines a hypergraph Hoffman bound \cite{filmus2021high} with the eigenvalues of the $q$-Kneser graph \cite{delsarte1976association}.
When $k\mid n$, we give an alternative proof of \eqref{eq:vector-space-emc-general}, combining a random sampling argument in the spirit of Katona's cycle method \cite{katona1972simple} with a field-extension trick. To the best of our knowledge, the field-extension trick first appeared in the work of Greene and Kleitman \cite{greene1978proof}, where it was used to prove a special case of the vector-space Erd\H{o}s--Ko--Rado theorem; see also \cite{chowdhury2010shadows} for a more recent application.
For the special case $n=(s+1)k$, \eqref{eq:vector-space-emc-n=(s+1)k} follows immediately from the upper bound in \eqref{eq:vector-space-emc-general} together with the construction $\mathcal{A}_q(k,s)$ in \eqref{eq:f1}. We emphasize that \eqref{eq:vector-space-emc-n=(s+1)k} is also implicit in a result of Chowdhury and Patk\'os \cite{chowdhury2010shadows} on $(s+1)$-wise intersecting families of vector spaces.

Finally, the proof of \eqref{eq:vector-space-emc-large-n} for large $n$ proceeds by induction, following the same general strategy as in \cite{bollob1976sets}. The same inductive argument also yields a Hilton--Milner-type stability result, which we discuss in the next subsection.

\subsection{A Hilton–Milner-type stability theorem}

Bollob\'as, Daykin, and Erd\H{o}s~\cite{bollob1976sets} proved a stability result for the Erd\H{o}s matching problem. To state their result, we first define the Hilton--Milner hypergraph $\mathcal{H}(n,k,s)$. Fix a vertex $1\in[n]$ and a set $E\in\binom{[n]}{k}$ with $1\notin E$, and define
$\mathcal{H}(n,k,1)=\{F\in\binom{[n]}{k}:1\in F,~F\cap E\neq\varnothing\}\cup\{E\}$. 
Starting from $\mathcal{H}(n-1,k,s-1)\subseteq\binom{[n-1]}{k}$, define recursively
$\mathcal{H}(n,k,s)=\{F\in\binom{[n]}{k}:n\in F\}\cup \mathcal{H}(n-1,k,s-1)$. 
It is easy to verify by induction on $s$ that $\nu(\mathcal{H}(n,k,s))\le s$.

Extending the Hilton--Milner theorem \cite{hilton1967some}, Bollob\'as, Daykin, and Erd\H{o}s~\cite{bollob1976sets} showed that for sufficiently large $n$, the hypergraph $\mathcal{H}(n,k,s)$ has the maximum number of edges among all hypergraphs $\mathcal{F}\subseteq\binom{[n]}{k}$ such that $\nu(\mathcal{F})\le s$ and $\mathcal{F}$ is not contained in any family isomorphic to $\mathcal{B}(n,k,s)$. 
This result has attracted considerable attention; see \cite{huang2017degree,frankl2020non,guo2026stability} for some recent developments.

Inspired by \cite{bollob1976sets}, we establish a stability result corresponding to \cref{thm:vemc} (iii). We now define the Hilton--Milner family $\mathcal{H}_q(n,k,s)\subseteq\qbinom{V}{k}$ in the vector-space setting. Fix a $1$-subspace $I\in\qbinom{V}{1}$ and a $k$-subspace $E\in\qbinom{V}{k}$ such that $I\cap E=\{0\}$, and define
\begin{equation*}
\mathcal{H}_q(n,k,1)=\left\{F\in \qbinom{V}{k}: I\le F,~F\cap E\neq \{0\}\right\}\cup \qbinom{I+E}{k}.
\end{equation*}

For $s\ge 2$, suppose that $\mathcal{H}_q(n-1,k,s-1)\subseteq\qbinom{V_{n-1}}{k}$ has been defined, where $\dim V_{n-1}=n-1$. Choose a $1$-dimensional vector space $I_n$ over $\mathbb F_q$, and let $V_n$ be the direct sum of $V_{n-1}$ and $I_n$. Let $\pi:V_n\to V_{n-1}$ be the canonical projection with kernel $I_n$. We then define
\begin{equation}\label{eq:H_q(n,k,s)}
  \mathcal{H}_q(n,k,s)=\left\{F\in\qbinom{V_n}{k}: I_n\le F\right\}\cup\left\{F\in\qbinom{V_n}{k}: \pi(F)\in\mathcal{H}_q(n-1,k,s-1)\right\}.  
\end{equation}
Note that the second family automatically consists of $k$-subspaces $F$ with $F\cap I_n=\{0\}$, since otherwise $\dim \pi(F)=k-1$ and hence $\pi(F)\notin\qbinom{V_{n-1}}{k}$. In Proposition \ref{prop:H_q(n,k,s)} below, we compute $|\mathcal{H}_q(n,k,s)|$ and show that $\nu_q(\mathcal{H}_q(n,k,s))\le s$. 

For $k=3$, there is one additional extremal family, defined as follows. Fix a $3$-subspace $E\in\qbinom{V}{3}$, and define
\begin{equation*}
\mathcal{H}_q'(n,3,1)=\left\{F\in\qbinom{V}{3}:\dim(F\cap E)\ge 2\right\}.
\end{equation*}
For $s\ge 2$, using the notation above, define recursively
\begin{equation}\label{eq:H_q'(n,3,s)}
    \mathcal{H}_q'(n,3,s)=\left\{F\in\qbinom{V_n}{3}: I_n\le F\right\}\cup\left\{F\in\qbinom{V_n}{3}: \pi(F)\in\mathcal{H}_q'(n-1,3,s-1)\right\}.
\end{equation}

In \cite{blokhuis2010hilton}, the authors characterized all largest families $\mathcal{F}\subseteq\qbinom{V}{k}$ satisfying $\nu_q(\mathcal{F})\le 1$ and not contained in any family isomorphic to $\mathcal{B}_q(n,k,1)$; see \cref{thm:HM} below. We extend their result to $\nu_q(\mathcal{F})\le s$ for every $s\ge 2$, and show that, for sufficiently large $n$, the families $\mathcal{H}_q(n,k,s)$ and $\mathcal{H}_q'(n,3,s)$ are extremal among all families $\mathcal{F}\subseteq\qbinom{V}{k}$ satisfying $\nu_q(\mathcal{F})\le s$ and not contained in any family isomorphic to $\mathcal{B}_q(n,k,s)$; see Proposition \ref{prop:H_q(n,k,s)} for properties of $\mathcal{H}_q(n,k,s)$ and $\mathcal{H}_q'(n,3,s)$.

\begin{theorem}\label{thm:q-HM}
Let $n$, $k$, and $s$ be positive integers with $n\ge (2s+1)k-s+3$ and $k\ge 2$, let $q$ be a prime power, and let $V$ be an $n$-dimensional vector space over $\mathbb{F}_q$. If $\mathcal{F}\subseteq\qbinom{V}{k}$ satisfies $\nu_q(\mathcal{F})\le s$ and is not contained in any family isomorphic to $\mathcal{B}_q(n,k,s)$, then $|\mathcal{F}|\le|\mathcal{H}_q(n,k,s)|$.

Moreover, if $k\neq3$, then equality holds if and only if $\mathcal{F}$ is isomorphic to $\mathcal{H}_q(n,k,s)$; if $k=3$, then equality holds if and only if $\mathcal{F}$ is isomorphic to either $\mathcal{H}_q(n,3,s)$ or $\mathcal{H}_q'(n,3,s)$.
\end{theorem}

Theorem~\ref{thm:q-HM} implies Theorem~\ref{thm:vemc}~{\rm (iii)}
when $n\ge (2s+1)k-s+3$. Indeed, let
$\mathcal F\subseteq\qbinom{V}{k}$ satisfy $\nu_q(\mathcal F)\le s$.
If $\mathcal F$ is contained in a family isomorphic to
$\mathcal B_q(n,k,s)$, then $|\mathcal F|\le |\mathcal{B}_q(n,k,s)|$. Otherwise,
Theorem~\ref{thm:q-HM} gives $|\mathcal F|\le |\mathcal{H}_q(n,k,s)|$. Since
$|\mathcal{H}_q(n,k,s)|<|\mathcal{B}_q(n,k,s)|$ for $n\ge (2s+1)k-s+3$, it follows that
$|\mathcal F|<|\mathcal{B}_q(n,k,s)|$ in this case. Hence
$m_q(n,k,s)=|\mathcal{B}_q(n,k,s)|$, and equality holds only for families isomorphic to
$\mathcal B_q(n,k,s)$.

\subsection{Cover-free families in vector spaces}

Cover-free families were introduced by Kautz and Singleton \cite{kautz2003nonrandom} and independently by Erd\H{o}s, Frankl, and F\"uredi \cite{erdos1982families,erdHos1985families}. They have numerous applications in design theory, coding theory, group testing, and cryptography. We refer the reader to the survey \cite{idalino2022survey} for a comprehensive introduction.

Formally, a family $\mathcal{F}\subseteq\binom{[n]}{r}$ is {\it $t$-cover-free} if no member is contained in the union of any other $t$ members; equivalently, there do not exist distinct sets $U,U_1,\ldots,U_t\in\mathcal{F}$ such that $U\subseteq U_1\cup\cdots\cup U_t$. Let $C_t(n,r)$ denote the maximum size of a $t$-cover-free family $\mathcal{F}\subseteq\binom{[n]}{r}$. Frankl and F\"uredi \cite{frankl1987colored} proved the following asymptotically sharp formula for $C_t(n,r)$ for all $t\ge 2$ and $r\ge 3$:
\begin{equation*}
    C_t(n,r)=(1+o(1))\cdot\frac{\binom{n}{k}}{\binom{r}{k}-m(r,k,s)},
\end{equation*}
where $o(1)\to 0$ as $n\to\infty$, and $0\le s\le t-1$ is the unique integer satisfying $r=k(s+1)+(k-1)(t-s-1)$. We set $m(r,k,0)=0$.

Recently, Shan and Zhou \cite{shan2026cover} introduced a vector-space analogue of $2$-cover-free families. For $n\ge r+1$, a family $\mathcal{F}\subseteq \qbinom{V}{r}$ is called {\it cover-free} if there do not exist three distinct subspaces $U,U_1,U_2\in\mathcal{F}$ such that $U\le (U_1\cap U)+(U_2\cap U)$. They proved that for positive integers $n$, $r$, and $k$ with $n\ge r+1$,
\begin{equation}\label{eq:vector-space-cff-t=2}
    |\mathcal{F}|\le
    \begin{cases}
        \dfrac{\qbinom{n-1}{k}_q}{\qbinom{r-1}{k}_q}, & \mbox{if } r=2k,\\[4pt]
        \dfrac{\qbinom{n}{k}_q}{\qbinom{r}{k}_q}, & \mbox{if } r=2k-1.
    \end{cases}
\end{equation}

They also characterized the extremal families, showing that these structures are closely related to $q$-Steiner systems. Here a family of $r$-subspaces $\mathcal{H}\subseteq\qbinom{V}{r}$ is called a {\it $q$-Steiner system} $S_q(n,r,k)$ on $V$ if every $k$-subspace $T\in\qbinom{V}{k}$ is contained in exactly one member $H\in\mathcal{H}$.

We consider the following generalization. For integers $n\ge r+1$ and $t\ge 2$, a family $\mathcal{F}\subseteq\qbinom{V}{r}$ is called {\it $t$-cover-free} if there do not exist $t+1$ distinct members $U,U_1,\ldots,U_t\in\mathcal{F}$ such that $U\le (U_1\cap U)+\cdots+(U_t\cap U)$.

Let $C_{q,t}(n,r)$ denote the maximum size of a $t$-cover-free family in $\qbinom{V}{r}$. We determine the asymptotic behavior of $C_{q,t}(n,r)$ in terms of $m_q(r,k,s)$ for every $t\ge 2$. We set $m_q(r,k,0)=0$. When $t=2$, our result is weaker than \eqref{eq:vector-space-cff-t=2} only up to a lower-order term.

\begin{theorem}\label{thm:vscff}
Let $k=\lceil r/t\rceil$, and let $0\le s\le t-1$ be the unique integer satisfying $r=(s+1)k+(t-s-1)(k-1)$. Then, for all fixed integers $r\ge 1$ and $t\ge 2$, and every prime power $q$, we have
\begin{equation*}
C_{q,t}(n,r)=(1+o(1))\cdot \frac{\qbinom{n}{k}_q}{\qbinom{r}{k}_q-m_q(r,k,s)},
\end{equation*}
where $o(1)\to 0$ as $n\to\infty$.

In particular, when $s=0$, equivalently $r=k+(t-1)(k-1)$, and $n$ is sufficiently large, we have
$C_{q,t}(n,r)\le \frac{\qbinom{n}{k}_q}{\qbinom{r}{k}_q}$, 
and equality holds if and only if there exists a $q$-Steiner system $S_q(n,r,k)$.
\end{theorem}

The proof of \cref{thm:vscff} is based on converting the cover-free condition into a local vector-space matching problem. For the upper bound, we count own $k$-subspaces, which are $k$-subspaces contained in precisely one member of $\mathcal{F}$.  
For the lower bound, fix an $r$-subspace $W$ and choose a family $\mathcal N\subseteq\qbinom{W}{k}$ with $|\mathcal N|=m_q(r,k,s)$ and $\nu_q(\mathcal N)\le s$. Let $\mathcal H:=\qbinom{W}{k}\setminus\mathcal N$. We use an induced packing lemma (see Lemma \ref{lem:induced-packing} below) to place asymptotically $\qbinom{n}{k}_q/|\mathcal H|$ pairwise disjoint copies of $\mathcal H$ inside $\qbinom{V}{k}$, supported on $r$-subspaces $W_1,W_2,\ldots$. The induced property ensures that, if some $W_j$ were covered by $t$ other members in the sense of the $t$-cover-free condition, then inside $W_j$ one would find $s+1$ members of a copy of $\mathcal N$ whose sum is direct, contradicting $\nu_q(\mathcal N)\le s$. Thus the supporting $r$-subspaces form a $t$-cover-free family of the required asymptotic size.

\paragraph{Organization.} The rest of this paper is organized as follows. In \cref{sec:pre}, we present some preliminaries. In \cref{sec:k=2,sec:Universal bound}, we prove \cref{thm:vemc} (i) and (ii), respectively. In \cref{sec:Vector space EMC for large $n$}, we prove \cref{thm:q-HM}, which essentially implies \cref{thm:vemc} (iii). In \cref{sec:$t$-cover families}, we present upper and lower bounds on $C_{q,t}(n,r)$ and prove \cref{thm:vscff}. 

\section{Preliminaries}\label{sec:pre}

For convenience, we denote $B_q(n,k,s):=|\mathcal{B}_q(n,k,s)|$ and $H_q(n,k,s):=|\mathcal{H}_q(n,k,s)|$. For subspaces $U$ and $W$ of a vector space, we write $U\le W$ to indicate that $U$ is a subspace of $W$. For subspaces $U_1,\ldots,U_m\le V$, we sometimes write $U_1\oplus\cdots\oplus U_m$ to emphasize that the sum is direct. For integers $0\le k\le n$, the Gaussian binomial coefficient $\qbinom{n}{k}_q$ is defined by
\begin{equation*}
\qbinom{n}{k}_q
=
\frac{(q^n-1)(q^{n-1}-1)\cdots(q^{n-k+1}-1)}
{(q^k-1)(q^{k-1}-1)\cdots(q-1)}.
\end{equation*}
It is symmetric, in the sense that $\qbinom{n}{k}_q=\qbinom{n}{n-k}_q$, and it satisfies the $q$-Pascal identities: for $1\le k\le n-1$,
\begin{equation*}
\qbinom{n}{k}_q
=
q^k\qbinom{n-1}{k}_q
+
\qbinom{n-1}{k-1}_q
=
\qbinom{n-1}{k}_q
+
q^{n-k}\qbinom{n-1}{k-1}_q.
\end{equation*}

We need the following two counting lemmas.

\begin{lemma}[{\cite[Lemma~9.3.2~(iii)]{brouwer1989distance}}]\label{lem:counting}
Let $X\le V$ with $\dim X=j$, and let $Z\le X$ with $\dim Z=d$. Then the number of subspaces
$Y\in \qbinom{V}{k}$ such that $X\cap Y=Z$ is
\begin{equation*}
q^{(k-d)(j-d)}\qbinom{n-j}{k-d}_q.
\end{equation*}
\end{lemma}

Taking $X=S$ with $\dim S=s$ and $Z=\{0\}$ in Lemma~\ref{lem:counting}, we obtain the number of $k$-subspaces of $V$ that intersect $S$ trivially:
\begin{equation*}
\left|\left\{F\in \qbinom{V}{k}: F\cap S=\{0\}\right\}\right|
=
q^{sk}\qbinom{n-s}{k}_q.
\end{equation*}
Therefore,
\begin{equation}\label{eq:B_q(n,k,s)}
B_q(n,k,s)=|\mathcal{B}_q(n,k,s)|
=
\qbinom{n}{k}_q-q^{sk}\qbinom{n-s}{k}_q.
\end{equation}

\begin{lemma}\label{lem:preimage count}
Let $V=U\oplus W$ be a direct sum of vector spaces over $\mathbb{F}_q$, with
$\dim U=t$ and $\dim W=n-t$. Let $\pi:V\to W$ be the canonical projection with kernel $U$. Then for every $k\le n-t$ and every $E\in\qbinom{W}{k}$,
\begin{equation*}
\left|\left\{F\in\qbinom{V}{k}:\pi(F)=E\right\}\right|=q^{kt}.
\end{equation*}
\end{lemma}

\begin{proof}
The proof is basic linear algebra and is omitted.
\end{proof}

Below we discuss the properties of $\mathcal H_q(n,k,s)$. We begin with the following fact. 

\begin{fact}\label{fact:H_q(n,k,1)}
Let $n,k$ be positive integers with $n> k\geq2$. Then 
\begin{equation*}
H_q(n,k,1)=\qbinom{n-1}{k-1}_q-q^{k(k-1)}\qbinom{n-k-1}{k-1}_q+q^k=B_q(n,k,0)+B_q(n-1,k-1,k)+q^k.
\end{equation*}
\end{fact}

\begin{proof}
The first equality can be found in \cite[(1.1)]{blokhuis2010hilton}. The second follows immediately from \eqref{eq:B_q(n,k,s)}. Here we use the convention $B_q(n,k,0)=0$.
\end{proof}

\begin{proposition}\label{prop:H_q(n,k,s)}
Let $n,k,s$ be positive integers with $n\ge k+s$ and $k\ge 2$. Then
\begin{itemize}
    \item [{\rm (i)}] 
    For $s\ge 2$, $H_q(n,k,s)=\qbinom{n-1}{k-1}_q+q^kH_q(n-1,k,s-1)$;
    \item [{\rm (ii)}] 
    $H_q(n,k,s)=B_q(n,k,s-1)+q^{k(s-1)}B_q(n-s,k-1,k)+q^{ks}$; 
    \item [{\rm (iii)}] $\nu_q(\mathcal H_q(n,k,s))\le s$, $\nu_q(\mathcal H_q'(n,3,s))\le s$;
    \item [{\rm (iv)}] $|\mathcal{H}_q'(n,3,s)|=|\mathcal{H}_q(n,3,s)|$.
\end{itemize}
\end{proposition}

\begin{proof}
For {\rm (i)}, let
\begin{equation*}
\mathcal X:=\left\{F\in\qbinom{V_n}{k}:I_n\le F\right\}
\qquad\text{and}\qquad
\mathcal Y:=\left\{F\in\qbinom{V_n}{k}:\pi(F)\in \mathcal H_q(n-1,k,s-1)\right\}.
\end{equation*}
Then by \eqref{eq:H_q(n,k,s)}, $\mathcal H_q(n,k,s)=\mathcal X\cup\mathcal Y$, and $\mathcal X\cap\mathcal Y=\varnothing$, since $I_n\le F$ implies $\dim \pi(F)=k-1$, whereas $\pi(F)\in\mathcal H_q(n-1,k,s-1)\subseteq\qbinom{V_{n-1}}{k}$ implies $\dim \pi(F)=k$. Hence
\begin{equation*}
H_q(n,k,s)=|\mathcal X|+|\mathcal Y|=\qbinom{n-1}{k-1}_q+q^kH_q(n-1,k,s-1),
\end{equation*}
where $|\mathcal X|=\qbinom{n-1}{k-1}_q$ and $|\mathcal Y|=q^kH_q(n-1,k,s-1)$ by Lemma~\ref{lem:preimage count}. 

To prove {\rm (ii)}, we apply induction on $s$. The case $s=1$ is Fact~\ref{fact:H_q(n,k,1)}. For $s\ge2$, assuming
\begin{equation*}
H_q(n-1,k,s-1)=B_q(n-1,k,s-2)+q^{k(s-2)}B_q(n-s,k-1,k)+q^{k(s-1)},
\end{equation*}
we obtain
\begin{equation*}
H_q(n,k,s)=\left(\qbinom{n-1}{k-1}_q+q^kB_q(n-1,k,s-2)\right)+q^{k(s-1)}B_q(n-s,k-1,k)+q^{ks}.
\end{equation*}
Since
\begin{equation*}
B_q(n,k,s-1)=\qbinom{n-1}{k-1}_q+q^kB_q(n-1,k,s-2),
\end{equation*}
by the definition of $B_q$ and the $q$-Pascal identity, the second equality in {\rm (ii)} follows.

For {\rm (iii)}, we again apply induction on $s$. The case $s=1$ is clear from the definition of $\mathcal H_q(n,k,1)$ and $\mathcal H_q'(n,3,1)$. In what follows, we give the proof of the induction step only for $\mathcal{H}_q(n,k,s)$; the proof for $\mathcal{H}'_q(n,3,s)$ is analogous. Assume $s\ge2$, and let $F_1,\ldots,F_{s+1}\in \mathcal H_q(n,k,s)$. If at least two of them lie in $\mathcal X$, then they both contain $I_n$, so they cannot form a direct sum. Thus at most one lies in $\mathcal X$. If exactly one, say $F_{s+1}\in\mathcal X$, then directness of $F_1,\ldots,F_{s+1}$ would force $(F_1+\cdots+F_s)\cap I_n=\{0\}$, so $\pi(F_1),\ldots,\pi(F_s)$ would form a direct sum in $\mathcal H_q(n-1,k,s-1)$, contradicting the induction hypothesis $\nu_q(\mathcal H_q(n-1,k,s-1))\le s-1$. Finally, if all $F_i$ lie in $\mathcal Y$, then for $K_i:=\pi(F_i)\in\mathcal H_q(n-1,k,s-1)$ we have
\begin{equation*}
\dim(K_1+\cdots+K_{s+1})\ge (s+1)k-1,
\end{equation*}
since $\ker\pi=I_n$ has dimension $1$. An elementary dimension count then shows that some $s$ of the $K_i$ form a direct sum, again contradicting the induction hypothesis. Thus $\nu_q(\mathcal H_q(n,k,s))\le s$.

For {\rm (iv)}, first observe that it is straightforward to verify by Fact \ref{fact:H_q(n,k,1)} and Lemma \ref{lem:counting} that $|\mathcal H_q(n,3,1)|=|\mathcal H'_q(n,3,1)|$. Since $\mathcal H_q(n,3,1)$ and $\mathcal H_q'(n,3,1)$ are defined by essentially the same recurrence \eqref{eq:H_q(n,k,s)} and \eqref{eq:H_q'(n,3,s)}, we obtain $|\mathcal{H}_q(n,3,s)|=|\mathcal{H}_q'(n,3,s)|$ for every $s\ge 2$.
\end{proof}

\section{Proof of \cref{thm:vemc} (i)}\label{sec:k=2}

In this section, we prove \cref{thm:vemc} (i). Our main tool is the following minimax theorem of Lov\'asz.

\begin{lemma}[{\cite[Theorem 2]{lovasz1980selecting}}]\label{lem: min-max}
For every $\mathcal{F}\subseteq\qbinom{V}{2}$, one has
\begin{equation*}
    \nu_q(\mathcal{F})=\min\left\{\dim(A)+\sum_{i=1}^m\left\lfloor\frac{\dim(A_i)-\dim(A)}{2}\right\rfloor\right\},
\end{equation*}
where the minimum is taken over all integers $m\ge 0$ and all subspaces $A,A_1,\dots,A_m\le V$ satisfying $A\le A_i$ for all $i\in[m]$, with the property that for every $F\in\mathcal{F}$, either $F\cap A\neq\{0\}$ or $F\le A_i$ for some $i\in[m]$.
\end{lemma}

We need two simple lemmas to prove Theorem~\ref{thm:vemc}~(i).

\begin{lemma}\label{lem: superadditivity}
For all $x,y\in\mathbb{Z}_{\ge 0}$,
\begin{equation*}
\qbinom{2x+1}{2}_q+\qbinom{2y+1}{2}_q\leq \qbinom{2(x+y)+1}{2}_q,
\end{equation*}
and the inequality is strict whenever $x,y>0$.
\end{lemma}

\begin{proof}
A direct calculation gives
\begin{equation*}
\qbinom{2(x+y)+1}{2}_q-\qbinom{2x+1}{2}_q-\qbinom{2y+1}{2}_q
=
\frac{(q^{2x}-1)(q^{2y}-1)(q^{2x+2y+1}+q^{2x+1}+q^{2y+1}-1)}{(q^2-1)(q-1)},
\end{equation*}
which is nonnegative for all $x,y\ge 0$, and positive whenever $x,y>0$.
\end{proof}

\begin{lemma}\label{lem: convexity}
For integers $a\in\{0,1,\ldots,s\}$, let
$h(a)=\qbinom{n}{2}_q-q^{2a}\qbinom{n-a}{2}_q+q^{2a}\qbinom{2(s-a)+1}{2}_q$. Then $h$ is strictly convex on $\{0,1,\ldots,s\}$ in the discrete sense.
\end{lemma}

\begin{proof}
A direct calculation gives
\begin{equation*}
h(a-1)+h(a+1)-2h(a)
=
\frac{q^{n+a-2}(q+1)(q-1)^2+q^{4s-2a-1}(q^2-1)^2}{(q^2-1)(q-1)}>0.
\end{equation*}
Hence $h$ is strictly convex.
\end{proof}

Now we are ready to prove Theorem~\ref{thm:vemc}~(i).

\begin{proof}[Proof of Theorem~\ref{thm:vemc}~(i)]
Let $\mathcal{F}\subseteq \qbinom{V}{2}$ satisfy $\nu_q(\mathcal{F})\le s$.
By Lemma~\ref{lem: min-max}, we may choose subspaces $A,A_1,\dots,A_m\le V$ such that, for every $F\in\mathcal{F}$, either $F\cap A\neq \{0\}$ or $F\le A_i$ for some $i\in[m]$, and
\begin{equation*}
\dim(A)+\sum_{i=1}^m\left\lfloor \frac{\dim(A_i)-\dim(A)}{2}\right\rfloor=\nu_q(\mathcal{F})\le s.
\end{equation*}
Set $a:=\dim(A)$, $d_i:=\dim(A_i)$ and $t_i:=\lfloor \frac{d_i-a}{2}\rfloor$. Then we have $a+\sum_{i=1}^m t_i\le s$, and $d_i-a\leq 2t_i+1$.
Let
\begin{equation*}
\mathcal{N}_A=\left\{T\in\qbinom{V}{2}: T\cap A\neq\{0\}\right\}
\quad\text{and}\quad
\mathcal{M}_i=\left\{T\in\qbinom{V}{2}: T\cap A=\{0\},~T\le A_i\right\}.
\end{equation*}
Then, by the choice of $A,A_1,\dots,A_m$, we have
\begin{equation*}
\begin{aligned}
|\mathcal{F}|
&\le |\mathcal{N}_A|+\sum_{i=1}^m |\mathcal{M}_i|\\
&= \qbinom{n}{2}_q-q^{2a}\qbinom{n-a}{2}_q+q^{2a}\sum_{i=1}^m \qbinom{d_i-a}{2}_q\\
&\leq \qbinom{n}{2}_q-q^{2a}\qbinom{n-a}{2}_q+q^{2a}\sum_{i=1}^m \qbinom{2t_i+1}{2}_q\\
&\le \qbinom{n}{2}_q-q^{2a}\qbinom{n-a}{2}_q+q^{2a}\qbinom{2\left(\sum_{i=1}^m t_i\right)+1}{2}_q\\
&\le \qbinom{n}{2}_q-q^{2a}\qbinom{n-a}{2}_q+q^{2a}\qbinom{2(s-a)+1}{2}_q, 
\end{aligned}
\end{equation*}
where the first equality follows from Lemma \ref{lem:counting}, the second and the last inequalities follow from the monotonicity of $\qbinom{x}{2}_q$, and the third inequality follows from Lemma~\ref{lem: superadditivity}.

By Lemma~\ref{lem: convexity}, $h(a)=\qbinom{n}{2}_q-q^{2a}\qbinom{n-a}{2}_q+q^{2a}\qbinom{2(s-a)+1}{2}_q$ is strictly convex in $a$, therefore
\begin{equation*}
|\mathcal{F}|\leq\max\{h(0),h(s)\}=\max\left\{\qbinom{2s+1}{2}_q,~\qbinom{n}{2}_q-q^{2s}\qbinom{n-s}{2}_q\right\}.
\end{equation*}
 
Now suppose that equality holds.
Then every inequality in the above chain must be an equality. In particular, by Lemma~\ref{lem: convexity} we must have $a\in\{0,s\}$, and equality in the last inequality implies $\sum_{i=1}^m t_i=s-a$.

If $a=0$, then $\sum_{i=1}^m t_i=s$. By Lemma~\ref{lem: superadditivity}, equality in the third inequality implies that at most one $t_i$ is positive. Hence there is a unique index $j$ such that $t_j=s$, while $t_i=0$ for all $i\ne j$. Since $d_i-a\le 2t_i+1$ for all $i$, we obtain $d_i\le 1$ for all $i\ne j$. Moreover, equality in $\qbinom{d_j-a}{2}_q\le \qbinom{2t_j+1}{2}_q$ implies $d_j-a=2t_j+1=2s+1$, so $d_j=2s+1$. Thus $\mathcal{N}_A=\varnothing$ and $\mathcal{M}_i=\varnothing$ for all $i\ne j$, and hence $\mathcal{F}\subseteq \qbinom{A_j}{2}$. Since $|\mathcal{F}|=\qbinom{2s+1}{2}_q=|\qbinom{A_j}{2}|$, it follows that $\mathcal{F}=\qbinom{A_j}{2}$, that is, $\mathcal{F}$ is isomorphic to $\mathcal{A}_q(2,s)$.

If $a=s$, then $\sum_{i=1}^m t_i=0$, so every $t_i=0$. Since $d_i-a\le 2t_i+1$ for all $i$, we have $d_i-a\le 1$, and therefore each $\mathcal{M}_i$ is empty. Thus $\mathcal{F}\subseteq \mathcal{N}_A$. Since $|\mathcal{F}|=|\mathcal{N}_A|=\qbinom{n}{2}_q-q^{2s}\qbinom{n-s}{2}_q$, we conclude that $\mathcal{F}=\mathcal{N}_A$, so $\mathcal{F}$ is isomorphic to $\mathcal{B}_q(n,2,s)$.
\end{proof}

\section{Proof of \cref{thm:vemc} (ii)}\label{sec:Universal bound}

In this section, we prove \cref{thm:vemc} (ii). We first prove the special case $k\mid n$ by an averaging argument. This yields the identity $m_q((s+1)k,k,s)=\qbinom{(s+1)k-1}{k}_q$. We then characterize the extremal family using an earlier result of Chowdhury and Patk\'os \cite{chowdhury2010shadows} on $r$-wise intersecting families of vector spaces. Finally, we prove the general case using a high-dimensional Hoffman bound \cite{filmus2021high}.

\subsection{The special case $k\mid n$}

\begin{proposition}\label{thm:k|n}
Suppose that $n\ge (s+1)k$ and $k\mid n$. Then
\begin{equation*}
m_q(n,k,s)\le \qbinom{s}{1}_{q^k}\qbinom{n-1}{k-1}_q.
\end{equation*}
\end{proposition}

\begin{proof}
Write $n=\ell k$, where $\ell\ge s+1$. Let $W=\mathbb{F}_{q^k}^{\,\ell}$ and $V=\mathbb{F}_q^n$, and fix an $\mathbb{F}_q$-linear isomorphism $\varPhi:W\to V$. Define
\begin{equation*}
\mathcal{B}:=\{\varPhi(L):L\le W,~\dim_{\mathbb{F}_{q^k}}L=1\}.
\end{equation*}
Then $\mathcal{B}\subseteq\qbinom{V}{k}$, since every $1$-dimensional $\mathbb{F}_{q^k}$-subspace of $W$ has dimension $k$ over $\mathbb{F}_q$, and $|\mathcal{B}|=\qbinom{\ell}{1}_{q^k}$.

We claim that every subfamily $\mathcal{G}\subseteq\mathcal{B}$ with $\nu_q(\mathcal{G})\le s$ satisfies $|\mathcal{G}|\le \qbinom{s}{1}_{q^k}$. Indeed, $\varPhi^{-1}(\mathcal{G})$ is a family of $1$-dimensional $\mathbb{F}_{q^k}$-subspaces of $W$, and $\nu_{q^k}(\varPhi^{-1}(\mathcal{G}))\le s$; otherwise, there would exist $L_1,\dots,L_{s+1}\in \varPhi^{-1}(\mathcal{G})$ whose sum is direct over $\mathbb{F}_{q^k}$. Then
\begin{equation*}
\dim_{\mathbb{F}_q}\bigl(\varPhi(L_1)+\cdots+\varPhi(L_{s+1})\bigr)
=
k\,\dim_{\mathbb{F}_{q^k}}(L_1+\cdots+L_{s+1})
=
k(s+1),
\end{equation*}
so $\varPhi(L_1),\dots,\varPhi(L_{s+1})$ form an $(s+1)$-matching in $\mathcal{G}$, a contradiction. Hence $|\mathcal{G}|=|\varPhi^{-1}(\mathcal{G})|\le m_{q^k}(\ell,1,s)=\qbinom{s}{1}_{q^k}$.

Now let $\mathcal{F}\subseteq\qbinom{V}{k}$ satisfy $\nu_q(\mathcal{F})\le s$, and choose $A\in\operatorname{GL}(V)$ uniformly at random. Since $\nu_q(A(\mathcal{F}))\le s$, applying the claim to $A(\mathcal{F})\cap\mathcal{B}$ gives $\mathbb{E}|A(\mathcal{F})\cap\mathcal{B}|\le \qbinom{s}{1}_{q^k}$. On the other hand, by linearity of expectation and the transitivity of the action of $\operatorname{GL}(V)$ on $\qbinom{V}{k}$, for each fixed $F\in\qbinom{V}{k}$ the random subspace $A(F)$ is uniformly distributed on $\qbinom{V}{k}$. Therefore
\begin{equation*}
\Pr[A(F)\in\mathcal{B}]
=
\frac{|\mathcal{B}|}{\qbinom{n}{k}_q}
=
\frac{\qbinom{\ell}{1}_{q^k}}{\qbinom{n}{k}_q}
=
\frac{1}{\qbinom{n-1}{k-1}_q},
\end{equation*}
and hence
\begin{equation*}
\mathbb{E}|A(\mathcal{F})\cap\mathcal{B}|
=
\sum_{F\in\mathcal{F}}\Pr[A(F)\in\mathcal{B}]
=
\frac{|\mathcal{F}|}{\qbinom{n-1}{k-1}_q}.
\end{equation*}
Comparing the two bounds, we obtain $|\mathcal{F}|\le \qbinom{s}{1}_{q^k}\qbinom{n-1}{k-1}_q$, as required.
\end{proof}

The above proposition yields $m_q((s+1)k,k,s)\le \qbinom{(s+1)k-1}{k}_q$, and this bound is attained by the construction $\mathcal{A}_q(k,s)$ in \eqref{eq:f1}. We now characterize the extremal family attaining equality. Recall that a family $\mathcal{G}\subseteq\qbinom{V}{k}$ is called {\it $r$-wise intersecting} if $\dim(\bigcap_{i=1}^r G_i)\ge 1$ for all $G_1,\ldots,G_r\in\mathcal{G}$. We shall use the following lemma, which is a special case of Theorem~1.5 in \cite{chowdhury2010shadows}.

\begin{lemma}[\cite{chowdhury2010shadows}]\label{lem:extremal-family}
Let $V$ be an $n$-dimensional vector space over $\mathbb{F}_q$. If $\mathcal{G}\subseteq\qbinom{V}{k}$ is $r$-wise intersecting and $(r-1)n\ge rk$, then $|\mathcal{G}|\le\qbinom{n-1}{k-1}_q$. Moreover, equality holds if and only if $\mathcal{G}=\{G\in\qbinom{V}{k}:L\le G\}$ for some fixed $1$-subspace $L\le V$, unless $r=2$ and $n=2k$.
\end{lemma}

\begin{proof}[Proof of \cref{thm:vemc} (ii), equality case]
Assume that $n=(s+1)k$, $\nu_q(\mathcal{F})\le s$, and $|\mathcal{F}|=\qbinom{(s+1)k-1}{k}_q$. We show that $\mathcal{F}$ is isomorphic to $\mathcal{A}_q(k,s)$, unless $s=1$ and $n=2k$. Since the statement is invariant under linear isomorphisms, we may assume that $V=\mathbb{F}_q^n$, equipped with the standard dot product $\langle x,y\rangle=\sum_{i=1}^n x_i y_i$.

For each subspace $F\le V$, let $F^\perp:=\{y\in V:\langle x,y\rangle=0\text{ for all }x\in F\}$, and set $\mathcal{F}^\perp:=\{F^\perp:F\in\mathcal{F}\}\subseteq\qbinom{V}{sk}$. Then $|\mathcal{F}^\perp|=|\mathcal{F}|=\qbinom{n-1}{k}_q=\qbinom{n-1}{sk-1}_q$.

We claim that $\mathcal{F}^\perp$ is $(s+1)$-wise intersecting. Indeed, if $F_1,\ldots,F_{s+1}\in\mathcal{F}$, then they do not form a direct sum, since $\nu_q(\mathcal{F})\le s$. As $\dim V=(s+1)k$, this is equivalent to saying that $F_1+\cdots+F_{s+1}\ne V$, and hence $(F_1+\cdots+F_{s+1})^\perp\ne\{0\}$. Since $(F_1+\cdots+F_{s+1})^\perp=F_1^\perp\cap\cdots\cap F_{s+1}^\perp$, the claim follows.

Now apply Lemma~\ref{lem:extremal-family} to $\mathcal{F}^\perp$ with $r=s+1$ and with $sk$ in place of $k$. Since $(r-1)n=s(s+1)k=(s+1)sk$, the lemma applies, and equality holds because $|\mathcal{F}^\perp|=\qbinom{n-1}{sk-1}_q$. Therefore, unless $r=2$ and $n=2sk$, there exists a fixed $1$-subspace $L\le V$ such that $\mathcal{F}^\perp=\{G\in\qbinom{V}{sk}:L\le G\}$. Since $r=s+1$, the exceptional case is exactly $s=1$ and $n=2k$.

It follows that $\mathcal{F}=\{F\in\qbinom{V}{k}:L\le F^\perp\}=\{F\in\qbinom{V}{k}:F\le L^\perp\}$. Since $L^\perp$ is an $((s+1)k-1)$-subspace of $V$, this shows that $\mathcal{F}$ is precisely the family of all $k$-subspaces contained in a fixed $((s+1)k-1)$-subspace. Hence $\mathcal{F}$ is isomorphic to $\mathcal{A}_q(k,s)$.
\end{proof}

\subsection{The general case}

The Hoffman bound \cite{hoffman2003eigenvalues} is a fundamental tool in spectral graph theory. It gives an upper bound on the independence number of a regular graph in terms of its eigenvalues: if $G$ is an $n$-vertex $d$-regular graph with eigenvalues $d=\lambda_1\ge \cdots \ge \lambda_n$, then $\alpha(G)\le \frac{-\lambda_n}{d-\lambda_n}\,n$. 

To prove the universal upper bound in \cref{thm:vemc} (ii), we will use a high-dimensional analogue of the Hoffman bound, also known as the {\it link bound}, due to Filmus, Golubev, and Lifshitz \cite{filmus2021high}. This bound applies to hypergraphs and yields an upper bound on their independence number.

We begin by introducing the necessary definitions.

\begin{definition}\label{def:hoff}
Let $\mathcal H$ be an $r$-uniform hypergraph with vertex set $V(\mathcal H)$ and edge set $E(\mathcal H)$.
\begin{enumerate}
    \item A set $I\subseteq V(\mathcal H)$ is called an independent set of $\mathcal H$ if ${I\choose r}\cap E(\mathcal H)=\varnothing$. The independence number of $\mathcal H$, denoted by $\alpha(\mathcal H)$, is the maximum size of an independent set in $\mathcal H$.

    \item For each $0\le i\le r$, let
    \begin{equation*}
    \mathcal H^{(i)}:=\left\{\sigma\in {V(\mathcal H)\choose i}: \sigma\subseteq e \text{ for some } e\in E(\mathcal H)\right\}.
    \end{equation*}

    \item For $\sigma\in \mathcal H^{(i)}$ with $0\le i\le r-1$, the link of $\sigma$ in $\mathcal H$ is the $(r-i)$-uniform hypergraph $\mathcal H_\sigma$ with vertex set
    \begin{equation*}
    V(\mathcal H_\sigma):=\left\{v\in V(\mathcal H)\setminus \sigma: \sigma\cup\{v\}\subseteq e \text{ for some } e\in E(\mathcal H)\right\}
    \end{equation*}
    and edge set
    \begin{equation*}
    E(\mathcal H_\sigma):=\left\{\tau\in {V(\mathcal H)\setminus \sigma\choose r-i}: \sigma\cup \tau\in E(\mathcal H)\right\}.
    \end{equation*}

    \item The skeleton graph of $\mathcal H$, denoted by $S(\mathcal H)$, is the simple graph on $V(\mathcal H)$ in which two distinct vertices $x,y\in V(\mathcal H)$ are adjacent if and only if there exists an edge $e\in E(\mathcal H)$ such that $\{x,y\}\subseteq e$.

    \item Let $G$ be a finite simple graph. The normalized adjacency matrix of $G$ is the matrix $T_G=(T_G(x,y))_{x,y\in V(G)}$ defined by
    \begin{equation*}
    T_G(x,y)=
    \begin{cases}
    1/d_G(y), & \text{if } \{x,y\}\in E(G),\\
    0, & \text{otherwise},
    \end{cases}
    \end{equation*}
    where $d_G(y)$ denotes the degree of $y$. We write $\lambda_{\min}(T_G)$ for the smallest eigenvalue of $T_G$.

    \item For each $0\le i\le r-2$, define
    \begin{equation*}
    \lambda_i(\mathcal H):=\min_{\sigma\in \mathcal H^{(i)}}\lambda_{\min}\bigl(T_{S(\mathcal H_\sigma)}\bigr).
    \end{equation*}
    In particular, $\lambda_0(\mathcal H)=\lambda_{\min}\bigl(T_{S(\mathcal H)}\bigr)$.
\end{enumerate}
\end{definition}
 
Below we present the link bound.

\begin{lemma}[{\cite[Theorem 1.4]{filmus2021high}}]\label{lem:link}
Let $\mathcal H$ be an $r$-uniform hypergraph. Suppose that $\mathcal H$ is regular, that is, every vertex of $\mathcal H$ is contained in the same positive number of edges. Suppose that for every $0\le i\le r-2$ and every $\sigma\in \mathcal H^{(i)}$, the skeleton graph $S(\mathcal H_\sigma)$ is regular, and the number
$|\{e\in E(\mathcal H): \sigma\cup\{x,y\}\subseteq e\}|$
is independent of the edge $\{x,y\}\in E(S(\mathcal H_\sigma))$. Then
\begin{equation*}
\alpha(\mathcal H)\le |V(\mathcal H)|
\left(
1-\frac{1}{(1-\lambda_0(\mathcal H))(1-\lambda_1(\mathcal H))\cdots(1-\lambda_{r-2}(\mathcal H))}
\right).
\end{equation*}
\end{lemma}

We will also need the spectrum of the $q$-Kneser graph. Let $q$ be a prime power, and let $n$ and $k$ be positive integers such that $n\ge 2k$. Let $V$ be an $n$-dimensional vector space over $\mathbb{F}_q$. The {\it $q$-Kneser graph} $\mathrm{KN}_q(n,k)$ is defined to be the graph whose vertices are the elements of $\qbinom{V}{k}$, with two vertices adjacent if and only if their intersection is $\{0\}$.

\begin{lemma}[{\cite[Theorem 10]{godsil2015erdos}}; {\cite[9.7.1 Corollary]{delsarte1976association}}]\label{lem:qKN}
The eigenvalues of the $q$-Kneser graph $\mathrm{KN}_q(n, k)$ are
\begin{equation*}
    \mu_j(n,k)=(-1)^jq^{\binom{j}{2}+k(k-j)}\qbinom{n-k-j}{k-j}_q, 
\end{equation*}
where $j = 0,\ldots, k$, with multiplicity $\qbinom{n}{j}_q-\qbinom{n}{j-1}_q$. 
\end{lemma}

Let $\otimes$ denote the standard {\it tensor product} of matrices. The following two facts are not hard to verify by definition, so we omit their proofs.

\begin{fact}
Let $A$ be an $n\times n$ matrix with eigenvalues $\lambda_1,\dots,\lambda_n$, and let $B$ be an $m\times m$ matrix with eigenvalues $\mu_1,\dots,\mu_m$. Then the eigenvalues of $A\otimes B$ are precisely $\lambda_i\mu_j$, where $i\in[n]$ and $j\in[m]$.
\end{fact}

Let $G[t]$ be the {\it $t$-fold blow-up} of $G$, that is, the graph obtained from $G$ by replacing each vertex $v\in V(G)$ with an independent set of size $t$, and replacing each edge of $G$ with a complete bipartite graph between the corresponding independent sets.

\begin{fact}\label{fact:eigenvalues}
Let $G$ be a finite graph without isolated vertices. Then $T_{G[t]}=T_G\otimes (\frac{1}{t}J_t)$, where $J_t$ is the $t\times t$ all-ones matrix. Since $\frac{1}{t}J_t$ has eigenvalue $1$ with multiplicity $1$ and eigenvalue $0$ with multiplicity $t-1$, the eigenvalues of $T_{G[t]}$ are precisely the eigenvalues of $T_G$, together with $(t-1)|V(G)|$ additional zero eigenvalues.
\end{fact}

\begin{proof}[Proof of Theorem~\ref{thm:vemc} {\rm (ii)}]
Let $\mathcal H$ be the $(s+1)$-uniform hypergraph defined on the vertex set $\qbinom{V}{k}$, where $V_1,\dots,V_{s+1}\in\qbinom{V}{k}$ form an edge if and only if their sum is direct. Then a family $\mathcal F\subseteq\qbinom{V}{k}$ satisfies $\nu_q(\mathcal F)\le s$ if and only if it is an independent set in $\mathcal H$, so $|\mathcal F|\le \alpha(\mathcal H)$. Moreover, $\mathcal H$ is regular, since $\mathrm{GL}(V)$ acts transitively on $\qbinom{V}{k}$ and preserves edges.

Fix $i\in\{0,1,\dots,s-1\}$ and let $\sigma=\{F_1,\dots,F_i\}\in \mathcal H^{(i)}$. Then $F_1,\dots,F_i$ form a direct sum. Put $U:=F_1+\cdots+F_i$, so $\dim U=ik$. The vertex set of $\mathcal H_\sigma$ consists precisely of those $A\in\qbinom{V}{k}$ with $A\cap U=\{0\}$, equivalently, $V(\mathcal H_\sigma)=\{A\in \qbinom{V}{k}:\dim(A+U)=(i+1)k\}$.

Let $V=U\oplus W$ be a direct sum, where $\dim W=n-ik$, and let $\pi:V\to W$ be the canonical projection with kernel $U$. Then $\pi$ sends $V(\mathcal H_\sigma)$ onto $\qbinom{W}{k}$, and for every $E\in\qbinom{W}{k}$, 
\begin{equation*}
V(\mathcal H_\sigma)\cap \pi^{-1}(E)
=
\{A\in V(\mathcal H_\sigma):\pi(A)=E\}
\end{equation*}
has size $q^{ik^2}$ by Lemma~\ref{lem:preimage count}.

Moreover, for $A,B\in V(\mathcal H_\sigma)$, the vertices $A$ and $B$ are adjacent in the skeleton graph $S(\mathcal H_\sigma)$ if and only if $\pi(A)\cap\pi(B)=\{0\}$. Indeed, $\{A,B\}\in E(S(\mathcal H_\sigma))$ if and only if $A+B+U$ is direct, which holds if and only if $\pi(A+B+U)=\pi(A+B)=\pi(A)+\pi(B)$ is direct. Therefore, $S(\mathcal H_\sigma)$ is the $q^{ik^2}$-fold blow-up of the $q$-Kneser graph $\mathrm{KN}_q(n-ik,k)$.

In particular, $S(\mathcal H_\sigma)$ is regular. Moreover, the subgroup of $\mathrm{GL}(V)$ fixing $F_1,\dots,F_i$ pointwise acts transitively on the ordered adjacent pairs of $S(\mathcal H_\sigma)$, so the number $|\{e\in E(\mathcal H):\sigma\cup\{A,B\}\subseteq e\}|$ is independent of the edge $\{A,B\}\in E(S(\mathcal H_\sigma))$. Thus the hypotheses of Lemma~\ref{lem:link} are satisfied.

We conclude that for every $i\in\{0,1,\ldots,s-1\}$ and every $\sigma\in\mathcal H^{(i)}$, the skeleton graph $S(\mathcal H_\sigma)$ is isomorphic to the $q^{ik^2}$-fold blow-up of the $q$-Kneser graph $\mathrm{KN}_q(n-ik,k)$. For brevity, write
\begin{equation*}
G_i:=\mathrm{KN}_q(n-ik,k).
\end{equation*}
By Lemma~\ref{lem:qKN}, the least eigenvalue of the adjacency matrix of $G_i$ is $\mu_1(n-ik,k)$, while the degree of $G_i$ is $\mu_0(n-ik,k)$. Hence the least eigenvalue of the normalized adjacency matrix $T_{G_i}$ is
\begin{equation*}
\lambda_{\min}(T_{G_i})
=
\frac{\mu_1(n-ik,k)}{\mu_0(n-ik,k)}
=
-\frac{q^{-k}\qbinom{k}{1}_q}{\qbinom{n-(i+1)k}{1}_q}.
\end{equation*}
Since $S(\mathcal H_\sigma)\cong G_i[q^{ik^2}]$, Fact \ref{fact:eigenvalues} implies that $T_{S(\mathcal H_\sigma)}$ has the same least eigenvalue as $T_{G_i}$. Therefore
\begin{equation*}
\lambda_i(\mathcal H)= -\frac{q^{-k}\qbinom{k}{1}_q}{\qbinom{n-(i+1)k}{1}_q}
\qquad\text{for }i=0,1,\ldots,s-1.
\end{equation*}

Now
\begin{equation*}
1-\lambda_i(\mathcal H)
=
1+\frac{q^{-k}\qbinom{k}{1}_q}{\qbinom{n-(i+1)k}{1}_q}
=
q^{-k}\frac{\qbinom{n-ik}{1}_q}{\qbinom{n-(i+1)k}{1}_q},
\end{equation*}
where the last equality follows from the identity
\begin{equation*}
\qbinom{n-ik}{1}_q=q^k\qbinom{n-(i+1)k}{1}_q+\qbinom{k}{1}_q.
\end{equation*}
Hence
\begin{equation*}
\prod_{i=0}^{s-1}(1-\lambda_i(\mathcal H))
=
q^{-ks}\frac{\qbinom{n}{1}_q}{\qbinom{n-sk}{1}_q}.
\end{equation*}
Applying Lemma~\ref{lem:link} to $\mathcal{H}$, we obtain
\begin{equation*}
\alpha(\mathcal H)
\le
\qbinom{n}{k}_q\left(1-\frac{q^{ks}\qbinom{n-sk}{1}_q}{\qbinom{n}{1}_q}\right).
\end{equation*}
Using $\qbinom{n}{1}_q=q^{ks}\qbinom{n-sk}{1}_q+\qbinom{sk}{1}_q$, this becomes
\begin{equation*}
\alpha(\mathcal H)
\le
\qbinom{n}{k}_q\cdot \frac{\qbinom{sk}{1}_q}{\qbinom{n}{1}_q}
=
\frac{\qbinom{sk}{1}_q}{\qbinom{k}{1}_q}\qbinom{n-1}{k-1}_q
=
\qbinom{s}{1}_{q^k}\qbinom{n-1}{k-1}_q.
\end{equation*}
Therefore every independent set of $\mathcal H$, and hence every family $\mathcal F\subseteq\qbinom{V}{k}$ with $\nu_q(\mathcal F)\le s$, has size at most $\qbinom{s}{1}_{q^k}\qbinom{n-1}{k-1}_q$.
\end{proof}

\section{Proof of \cref{thm:q-HM} and \cref{thm:vemc} (iii)}
\label{sec:Vector space EMC for large $n$}
\label{sec:Stability}

For a family $\mathcal{F}\subseteq \qbinom{V}{k}$ and a subspace $I\le V$, define $\mathcal{F}[I]:=\{F\in\mathcal{F}: I\le F\}$ to be the subfamily of members of $\mathcal{F}$ containing $I$.

\begin{lemma}\label{lem: matching}
Let $n,k,s$ be positive integers satisfying $n\ge (s+1)k$ and let $V$ be an $n$-dimensional vector space over $\mathbb{F}_q$. Let $\mathcal{F}\subseteq \qbinom{V}{k}$ be a family of $k$-subspaces satisfying $\nu_q (\mathcal{F})\le s$. Then the following hold.
\begin{itemize}
    \item [{\rm (i)}] There exists at least one $I\in\qbinom{V}{1}$ such that $|\mathcal{F}[I]|\ge\frac{|\mathcal{F}|}{\qbinom{sk}{1}_q}$.
    \item [{\rm (ii)}] For every $I\in \qbinom{V}{1}$, if there exist $F_1,\ldots,F_s\in\mathcal{F}$ such that $U:=F_1+\cdots+F_s$ is direct and $I\cap U=\{0\}$, then $|\mathcal{F}[I]|\le B_q(n-1,k-1,sk)$.
\end{itemize}
\end{lemma}

\begin{proof}
For (i), let $F_1,\ldots,F_i$ be a maximal subfamily of $\mathcal{F}$ such that $U:=F_1+\cdots+F_i$ is direct. Then $i\le s$ and $\dim U\le sk$. By the maximality of $U$, every member of $\mathcal{F}$ intersects $U$ nontrivially, so by averaging, there is at least one 1-dimensional subspace of $U$ that is contained in at least $\frac{|\mathcal{F}|}{\qbinom{sk}{1}_q}$ members of $\mathcal{F}$.  

For (ii), let $\mathcal{G}:=\{F\in\qbinom{V}{k}:I\le F,~U\cap F=\{0\}\}$. Applying Lemma~\ref{lem:counting} with $X:=U\oplus I$ and $Z:=I$ gives that $|\mathcal{G}|=q^{(k-1)sk}\qbinom{n-1-sk}{k-1}_q$. Moreover, $\mathcal{F}[I]\cap\mathcal{G}=\varnothing$, since otherwise $\nu_q(\mathcal{F})\ge s+1$, a contradiction. Therefore, 
\begin{equation*}
    |\mathcal{F}[I]|\le\qbinom{n-1}{k-1}_q-|\mathcal{G}|=\qbinom{n-1}{k-1}_q-q^{(k-1)sk}\qbinom{n-1-sk}{k-1}_q=B_q(n-1,k-1,sk), 
\end{equation*}
where the last equality follows by \eqref{eq:B_q(n,k,s)}.
\end{proof}

Our proof of Theorem~\ref{thm:q-HM} proceeds by induction on $s$. The base case $s=1$ builds on the following result.  

\begin{theorem}[{\cite[Theorem 1.4]{blokhuis2010hilton}}]\label{thm:HM}
    Suppose either $q\ge 3$ and $n\ge 2k+1$, or $q=2$ and $n\ge 2k+2$. For any family $\mathcal{F}\subseteq\qbinom{V}{k}$ satisfying $\nu_q(\mathcal{F})\le 1$ and not contained in any family isomorphic to $\mathcal{B}_q(n,k,1)$, we have $|\f|\le H_q(n,k,1)$. Equality holds if and only if $\f$ is isomorphic to either $\mathcal H_q(n,k,1)$ or $H^\prime_q(n,3,1)$ (when $k=3$).
\end{theorem}

\begin{proof}[Proof of Theorem~\ref{thm:q-HM}]
We prove the theorem by induction on $s$. For a $1$-subspace $I\in\qbinom{V}{1}$, let $\pi_I:V\to V/I,~x\mapsto x+I$ be the quotient map, and write 
\begin{equation*}
\pi_I(\mathcal{F}\setminus\mathcal{F}[I])
:=\{\pi_I(F):F\in\f\setminus\f[I]\}\subseteq \qbinom{V/I}{k}. 
\end{equation*} 
Notice that if $F\in\f\setminus\f[I]$, then $F\cap I=\{0\}$, and hence $\dim\pi_I(F)=k$.

The case $s=1$ follows from Theorem~\ref{thm:HM}, since our assumption $n\ge (2s+1)k-s+3$ gives $n\ge 3k+2$. Now assume $s\ge2$ and that the theorem has been proved for $s-1$.

We need the following simple observation. If, for some $I\in\qbinom{V}{1}$, the quotient family $\pi_I(\mathcal{F}\setminus\mathcal{F}[I])$ is contained in a family isomorphic to $\mathcal B_q(n-1,k,s-1)$ in $V/I$, then $\f$ is contained in a family isomorphic to $\mathcal B_q(n,k,s)$ in $V$. Indeed, in this case there is an $(s-1)$-subspace $\overline S\le V/I$ such that every member of $\pi_I(\mathcal{F}\setminus\mathcal{F}[I])$ intersects $\overline S$ nontrivially. Let $S:=\pi_I^{-1}(\overline S)$. Then $\dim S=s$ and $I\le S$. If $F\in\f[I]$, then $F\cap S\neq\{0\}$ because $I\le F\cap S$; if $F\in\f\setminus\f[I]$, then $\pi_I(F)\cap\overline S\neq\{0\}$, which implies $F\cap S\neq\{0\}$. Thus $\f\subseteq\{F\in\qbinom{V}{k}:F\cap S\neq\{0\}\}$, a contradiction. Therefore, whenever the induction hypothesis is applied to $\pi_I(\mathcal{F}\setminus\mathcal{F}[I])$, this quotient family is not contained in any family isomorphic to $\mathcal B_q(n-1,k,s-1)$.

We also have $\nu_q(\pi_I(\mathcal{F}\setminus\mathcal{F}[I]))\le s$ for every $I\in\qbinom{V}{1}$. Otherwise, $s+1$ members of $\pi_I(\mathcal{F}\setminus\mathcal{F}[I])$ whose sum is direct could be lifted to $s+1$ members of $\f$ whose sum is direct, contradicting $\nu_q(\f)\le s$.

\medskip
\noindent\textbf{Case 1.} Suppose that there exists a $1$-subspace $I\in\qbinom{V}{1}$ such that $\nu_q(\pi_I(\mathcal{F}\setminus\mathcal{F}[I]))\le s-1$. By the observation above and the induction hypothesis, $|\pi_I(\mathcal{F}\setminus\mathcal{F}[I])|\le H_q(n-1,k,s-1)$. Moreover, by Lemma~\ref{lem:preimage count}, for every $K\in\pi_I(\mathcal{F}\setminus\mathcal{F}[I])$, the preimage $\{F\in\qbinom{V}{k}:F\cap I=\{0\},~\pi_I(F)=K\}$
has size $q^k$. Hence $|\f\setminus\f[I]|\le q^k|\pi_I(\mathcal{F}\setminus\mathcal{F}[I])|\le q^kH_q(n-1,k,s-1)$. Since $|\f[I]|\le\qbinom{n-1}{k-1}_q$, Proposition~\ref{prop:H_q(n,k,s)} gives
\begin{equation*}
|\f|\le \qbinom{n-1}{k-1}_q+q^kH_q(n-1,k,s-1)=H_q(n,k,s).
\end{equation*}

It remains to discuss the equality. If $|\f|=H_q(n,k,s)$, then all inequalities above are equalities. Thus $\f[I]=\{F\in\qbinom{V}{k}:I\le F\}$, the image $\pi_I(\mathcal{F}\setminus\mathcal{F}[I])$ has size $H_q(n-1,k,s-1)$, and for every $K\in\pi_I(\mathcal{F}\setminus\mathcal{F}[I])$, the whole preimage
$\{F\in\qbinom{V}{k}:F\cap I=\{0\},~\pi_I(F)=K\}$
is contained in $\f$. By the induction hypothesis, if $k\neq 3$, then $\pi_I(\mathcal{F}\setminus\mathcal{F}[I])$ is isomorphic to $\mathcal H_q(n-1,k,s-1)$; if $k=3$, then $\pi_I(\mathcal{F}\setminus\mathcal{F}[I])$ is isomorphic to either $\mathcal H_q(n-1,3,s-1)$ or $\mathcal H_q'(n-1,3,s-1)$. 
Since the whole preimage of each member of $\pi_I(\mathcal{F}\setminus\mathcal{F}[I])$ is contained in $\f$, we have
\begin{equation*}
\f=\left\{F\in\qbinom{V}{k}:I\le F\right\}\cup\left\{F\in\qbinom{V}{k}:F\cap I=\{0\},~\pi_I(F)\in\pi_I(\mathcal{F}\setminus\mathcal{F}[I])\right\}.
\end{equation*}
By the recursive definitions \eqref{eq:H_q(n,k,s)} and \eqref{eq:H_q'(n,3,s)}, it follows that $\f$ is isomorphic to $\mathcal H_q(n,k,s)$ when $k\neq 3$, and is isomorphic to either $\mathcal H_q(n,3,s)$ or $\mathcal H_q'(n,3,s)$ when $k=3$.

\medskip
\noindent\textbf{Case 2.} Suppose that for every $1$-subspace $I\in\qbinom{V}{1}$, we have $\nu_q(\pi_I(\mathcal{F}\setminus\mathcal{F}[I]))=s$. We show that this case cannot occur when $|\f|\ge H_q(n,k,s)$.

By Lemma~\ref{lem: matching}~{\rm (i)}, there exists $I\in\qbinom{V}{1}$ such that $|\f[I]|\ge |\f|/\qbinom{sk}{1}_q$. Since we are in Case 2, there exist $K_1,\ldots,K_s\in\pi_I(\mathcal{F}\setminus\mathcal{F}[I])$ whose sum is direct in $V/I$. Choose $F_1,\ldots,F_s\in\f\setminus\f[I]$ with $\pi_I(F_i)=K_i$ for all $i\in[s]$, and set $U:=F_1+\cdots+F_s$. Since $K_1+\cdots+K_s$ is direct, we have $\dim\pi_I(U)=sk$. As $\dim U\le sk$, it follows that $\dim U=sk$ and $U\cap I=\{0\}$. Thus $F_1,\ldots,F_s$ form a direct sum and $I\cap U=\{0\}$.

Applying Lemma~\ref{lem: matching}~{\rm (ii)} to this $I$ and $U$, we obtain $|\f[I]|\le B_q(n-1,k-1,sk)=\qbinom{n-1}{k-1}_q-q^{sk(k-1)}\qbinom{n-sk-1}{k-1}_q$. Therefore $|\f|/\qbinom{sk}{1}_q\le B_q(n-1,k-1,sk)$. If $|\f|\ge H_q(n,k,s)$, then $H_q(n,k,s)/\qbinom{sk}{1}_q\le B_q(n-1,k-1,sk)$, contradicting Proposition \ref{prop:H-parameter}. Hence in Case 2 we must have $|\f|<H_q(n,k,s)$.

Combining the two cases gives $|\f|\le H_q(n,k,s)$. The equality cases are precisely those obtained in Case 1. Conversely, the families $\mathcal H_q(n,k,s)$ and, when $k=3$, $\mathcal H_q'(n,3,s)$ satisfy $\nu_q\le s$, are not contained in any family isomorphic to $\mathcal B_q(n,k,s)$, and have size $H_q(n,k,s)$ by Proposition~\ref{prop:H_q(n,k,s)}. This completes the proof.
\end{proof}

The proof of Theorem~\ref{thm:vemc}~{\rm (iii)} follows the same induction argument as the proof of Theorem~\ref{thm:q-HM}, with $B_q(n,k,s)$ in place of $H_q(n,k,s)$.
So here we only sketch the proof.
We use the following vector-space Erd\H{o}s--Ko--Rado theorem, due to Hsieh~\cite{hsieh1975intersection} and Frankl and Wilson~\cite{frankl1986erdos}, as the base case of the induction.

\begin{theorem}[{\cite[Theorem~1]{frankl1986erdos}}]\label{thm:ekr for vs}
Let $V$ be an $n$-dimensional vector space over $\mathbb F_q$, and let
$\ma{F}\subseteq\qbinom{V}{k}$ satisfy $F\cap F'\neq \{0\}$ for all $F,F'\in \ma{F}$.  If $n\ge 2k$, then $|\ma{F}|\le \qbinom{n-1}{k-1}_q$. Moreover, if $n>2k$, equality holds if and only if $\ma{F}=\{F\in \qbinom{V}{k}: I\le F\}$ for some fixed $1$-subspace $I\le V$.
\end{theorem}

\begin{proof}[Proof of Theorem~\ref{thm:vemc}~{\rm (iii)}]
The case $s=1$ follows from \cref{thm:ekr for vs}. Assume
$s\ge2$, and let $\mathcal F\subseteq\qbinom{V}{k}$ satisfy
$\nu_q(\mathcal F)\le s$. For a $1$-subspace $I\le V$, let $\pi_I:V\to V/I,~x\mapsto x+I$ be the quotient map.

If there exists $I\in\qbinom{V}{1}$ such that
$\nu_q(\pi_I(\mathcal{F}\setminus\mathcal{F}[I]))\le s-1$, then the induction hypothesis gives
$|\pi_I(\mathcal{F}\setminus\mathcal{F}[I])|\le B_q(n-1,k,s-1)$. Since each preimage under the
quotient map has size at most $q^k$, and since
$|\mathcal F[I]|\le\qbinom{n-1}{k-1}_q$, we obtain
\begin{equation*}
|\mathcal F|\le \qbinom{n-1}{k-1}_q+q^kB_q(n-1,k,s-1)=B_q(n,k,s).
\end{equation*}
Moreover, equality forces $\mathcal F[I]=\{F\in\qbinom{V}{k}:I\le F\}$, the
image $\pi_I(\mathcal{F}\setminus\mathcal{F}[I])$ to be extremal for $B_q(n-1,k,s-1)$, and
the preimage of every member of $\pi_I(\mathcal{F}\setminus\mathcal{F}[I])$ to be full. Hence
$\mathcal F$ is isomorphic to $\mathcal B_q(n,k,s)$.

It remains to exclude the case in which
$\nu_q(\pi_I(\mathcal{F}\setminus\mathcal{F}[I]))=s$ for every $I\in\qbinom{V}{1}$. By
Lemma~\ref{lem: matching}~{\rm (i)}, there exists $I\in\qbinom{V}{1}$ such
that $|\mathcal F[I]|\ge |\mathcal F|/\qbinom{sk}{1}_q$. As in the proof of
Theorem~\ref{thm:q-HM}, the assumption
$\nu_q(\pi_I(\mathcal{F}\setminus\mathcal{F}[I]))=s$ allows us to apply
Lemma~\ref{lem: matching}~{\rm (ii)}, and therefore
\begin{equation*}
\frac{|\mathcal F|}{\qbinom{sk}{1}_q}\le B_q(n-1,k-1,sk).
\end{equation*}
If $|\mathcal F|\ge B_q(n,k,s)$, then
\begin{equation*}
\frac{B_q(n,k,s)}{\qbinom{sk}{1}_q}\le B_q(n-1,k-1,sk),
\end{equation*}
contradicting Proposition \ref{prop:B-parameter}.
\end{proof}

\section{Proof of \cref{thm:vscff}}\label{sec:$t$-cover families}

\subsection{Proof of the upper bound}

In this subsection, we prove the upper bound in \cref{thm:vscff}, namely
$C_{q,t}(n,r)\le \qbinom{n}{k}_q/(\qbinom{r}{k}_q-m_q(r,k,s))$ for sufficiently large $n$. 

For a family $\mathcal{F}\subseteq\qbinom{V}{r}$ and a subspace $T\le V$, we say that $T$ is an {\it own subspace} of $F\in\mathcal{F}$ if $F$ is the unique member of $\mathcal{F}$ containing $T$. For each positive integer $a$, let $\mathcal{D}_a(F)$ denote the family of all own $a$-subspaces of $F$.

\begin{proof}[Proof of \cref{thm:vscff}, the upper bound]
We may assume that $|\mathcal{F}|\ge t+1$, since otherwise the desired bound is trivial for all sufficiently large $n$. Since $q,r,t$ are fixed, we may also assume that $n$ is large enough so that $\qbinom{n-k+1}{1}_q/\qbinom{k}{k-1}_q\ge \qbinom{r}{k}_q-m_q(r,k,s)$.

Let $\mathcal{F}\subseteq\qbinom{V}{r}$ be a $t$-cover-free family. Split $\mathcal{F}$ into
$\mathcal{F}_1:=\{F\in\mathcal{F}:\mathcal{D}_{k-1}(F)\neq\varnothing\}$ and $\mathcal{F}_0:=\mathcal{F}\setminus\mathcal{F}_1$. Since every own $(k-1)$-subspace is contained in a unique member of $\mathcal{F}$, the families $\mathcal{D}_{k-1}(F)$, $F\in\mathcal{F}_1$, are pairwise disjoint. In particular,
\begin{equation*}
\left|\bigcup_{F\in\mathcal{F}_1}\mathcal{D}_{k-1}(F)\right|
=
\sum_{F\in\mathcal{F}_1}|\mathcal{D}_{k-1}(F)|
\ge |\mathcal{F}_1|.
\end{equation*}

For each $F\in\mathcal{F}_1$, define
$\mathcal{A}_F:=\{S\in\qbinom{V}{k}:T\le S\text{ for some }T\in\mathcal{D}_{k-1}(F)\}$. Counting pairs $(S,T)$ with $T\le S$, where $S\in\bigcup_{F\in\mathcal{F}_1}\mathcal{A}_F$ and $T\in\bigcup_{F\in\mathcal{F}_1}\mathcal{D}_{k-1}(F)$, gives
\begin{equation*}
\left|\bigcup_{F\in\mathcal{F}_1}\mathcal{A}_F\right|
\ge
\frac{\qbinom{n-k+1}{1}_q}{\qbinom{k}{k-1}_q}
\left|\bigcup_{F\in\mathcal{F}_1}\mathcal{D}_{k-1}(F)\right|
\ge
\frac{\qbinom{n-k+1}{1}_q}{\qbinom{k}{k-1}_q}|\mathcal{F}_1|.
\end{equation*}
Indeed, each member of $\cup_{F\in\mathcal{F}_1}\mathcal{D}_{k-1}(F)$ is contained in exactly $\qbinom{n-k+1}{1}_q$ $k$-subspaces of $V$, while each member of $\cup_{F\in\mathcal{F}_1}\mathcal{A}_F$ contains at most $\qbinom{k}{k-1}_q$ members of $\cup_{F\in\mathcal{F}_1}\mathcal{D}_{k-1}(F)$.

Moreover, $\bigcup_{F\in\mathcal{F}_1}\mathcal{A}_F$ is disjoint from $\bigcup_{F\in\mathcal{F}_0}\mathcal{D}_k(F)$. Indeed, if $S$ belonged to both, then $S$ would contain an own $(k-1)$-subspace $T$ of some member $F_1\in\mathcal{F}_1$, and $S$ would itself be an own $k$-subspace of some member $F_0\in\mathcal{F}_0$. Since $S\le F_0$, we would have $T\le F_0$, which implies that $F_0=F_1$, a contradiction. Also, the families $\mathcal{D}_k(F)$, $F\in\mathcal{F}_0$, are pairwise disjoint. Therefore
\begin{equation*}
\qbinom{n}{k}_q
\ge
\left|\bigcup_{F\in\mathcal{F}_1}\mathcal{A}_F\right|
+
\sum_{F\in\mathcal{F}_0}|\mathcal{D}_k(F)|
\ge
\frac{\qbinom{n-k+1}{1}_q}{\qbinom{k}{k-1}_q}|\mathcal{F}_1|
+
\sum_{F\in\mathcal{F}_0}|\mathcal{D}_k(F)|.
\end{equation*}

\begin{claim}\label{claim:cff-UB-F_0}
For every $F\in\mathcal{F}_0$, we have
\begin{equation*}
|\mathcal{D}_k(F)|\ge \qbinom{r}{k}_q-m_q(r,k,s).
\end{equation*}
\end{claim}

Assuming the claim, the preceding inequality and the choice of $n$ imply
\begin{equation*}
\qbinom{n}{k}_q
\ge
\left(\qbinom{r}{k}_q-m_q(r,k,s)\right)|\mathcal{F}_1|
+
\left(\qbinom{r}{k}_q-m_q(r,k,s)\right)|\mathcal{F}_0|
=
\left(\qbinom{r}{k}_q-m_q(r,k,s)\right)|\mathcal{F}|.
\end{equation*}
This gives $|\mathcal{F}|\le \qbinom{n}{k}_q/(\qbinom{r}{k}_q-m_q(r,k,s))$, as desired.

It remains to prove the claim. Fix $F\in\mathcal{F}_0$, and let $\mathcal{G}:=\qbinom{F}{k}\setminus\mathcal{D}_k(F)$ be the family of non-own $k$-subspaces of $F$. It is enough to show that $\nu_q(\mathcal{G})\le s$. Indeed, this would give $|\mathcal{G}|\le m_q(r,k,s)$, and hence $|\mathcal{D}_k(F)|=\qbinom{r}{k}_q-|\mathcal{G}|\ge \qbinom{r}{k}_q-m_q(r,k,s)$.

Suppose, to the contrary, that $\nu_q(\mathcal{G})\ge s+1$. Then there exist $T_1,\ldots,T_{s+1}\in\mathcal{G}$ whose sum is direct. Since each $T_i$ is non-own, for every $i\in[s+1]$ there exists $A_i\in\mathcal{F}\setminus\{F\}$ such that $T_i\le A_i$. As $\dim F=r=(s+1)k+(t-s-1)(k-1)$, we may write $F=T_1\oplus\cdots\oplus T_{s+1}\oplus S$, where $\dim S=(t-s-1)(k-1)$. Decompose $S$ as $S=S_1\oplus\cdots\oplus S_{t-s-1}$, with $\dim S_j=k-1$ for every $j$; when $t-s-1=0$, this sum is empty. Since $F\in\mathcal{F}_0$, no $(k-1)$-subspace of $F$ is own. Hence, for every $j$, there exists $B_j\in\mathcal{F}\setminus\{F\}$ such that $S_j\le B_j$.

Let $C_1,\ldots,C_m$ be the distinct members among $A_1,\ldots,A_{s+1},B_1,\ldots,B_{t-s-1}$. Then $m\le t$, all $C_\ell$ are distinct from $F$, and
\begin{equation*}
F\le \sum_{\ell=1}^m(F\cap C_\ell).
\end{equation*}
If $m<t$, then $|\mathcal{F}|\ge t+1$ allows us to choose further distinct members $C_{m+1},\ldots,C_t\in\mathcal{F}\setminus\{F,C_1,\ldots,C_m\}$. Thus $F\le \sum_{\ell=1}^t(F\cap C_\ell)$, contradicting that $\mathcal{F}$ is $t$-cover-free. This proves the claim and completes the proof.
\end{proof}

\subsection{Proof of the lower bound}

We now prove the lower bound in \cref{thm:vscff}. We first recall the notion of packing designs in vector spaces. Let $V$ be an $n$-dimensional vector space over $\mathbb{F}_q$. A family $\mathcal{P}\subseteq\qbinom{V}{k}$ is called an {\it $[n,k,t]_q$-packing design} if every $t$-subspace of $V$ is contained in at most one member of $\mathcal{P}$.

\begin{lemma}[{\cite[Theorem 2]{blackburn2012asymptotic}}]\label{lem:near-optimal-packing}
Fix integers $q,k,t$, where $q$ is a prime power and $1\le t\le k$. Then for every $\varepsilon>0$, there exists $n_0=n_0(q,k,t,\varepsilon)$ such that, for all $n\ge n_0$, there exists an $[n,k,t]_q$-packing design of size at least $(1-\varepsilon)\qbinom{n}{t}_q/\qbinom{k}{t}_q$.
\end{lemma}

We say that a family $\mathcal{F}'\subseteq\qbinom{W'}{k}$ is a copy of a family $\mathcal{F}\subseteq\qbinom{W}{k}$ if there exists a linear isomorphism $\varPhi:W\to W'$ such that $\mathcal{F}'=\{\varPhi(F):F\in\mathcal{F}\}$.

We will need the following key lemma.

\begin{lemma}\label{lem:induced-packing}
Fix positive integers $e, k, r$ with $e\ge2$ and $r\ge k+1$, and a prime power $q$. For every $\gamma>0$, there exists $n_0=n_0(e,k,q,r,\gamma)$ such that whenever $n\ge n_0$, the following holds. For every $n$-dimensional vector space $V$ over $\mathbb{F}_q$, every $r$-subspace $W\le V$, and every family $\mathcal{F}\subseteq\qbinom{W}{k}$ of size $e$, there exist pairwise disjoint copies $\mathcal{F}_1,\ldots,\mathcal{F}_m$ of $\mathcal{F}$ in $\qbinom{V}{k}$, with $\mathcal{F}_i\subseteq\qbinom{W_i}{k}$ for some $r$-subspace $W_i\le V$, such that $m\ge (1-\gamma)\qbinom{n}{k}_q/e$, and moreover:
\begin{itemize}
    \item[{\rm (i)}] $\dim(W_i\cap W_j)\le k$ for all $i\ne j$;
    \item[{\rm (ii)}] if $\dim(W_i\cap W_j)=k$ and $A=W_i\cap W_j$, then $A\notin\mathcal{F}_i$ and $A\notin\mathcal{F}_j$.
\end{itemize}
\end{lemma}

We postpone the proof of Lemma~\ref{lem:induced-packing} and now prove the lower bound in \cref{thm:vscff}.

\begin{proof}[Proof of \cref{thm:vscff}, the lower bound]
Fix $\gamma>0$. It suffices to construct, for all sufficiently large $n$, a $t$-cover-free family in $\qbinom{V}{r}$ of size at least
\begin{equation*}
(1-\gamma)\frac{\qbinom{n}{k}_q}{\qbinom{r}{k}_q-m_q(r,k,s)}.
\end{equation*}
The case $k=r=1$ is immediate, since the family of all $1$-subspaces of $V$ is $t$-cover-free and has size $\qbinom{n}{1}_q$. We may therefore assume $r\ge k+1$.

Let $W\le V$ be an $r$-subspace. Choose $\mathcal{N}\subseteq\qbinom{W}{k}$ with $|\mathcal{N}|=m_q(r,k,s)$ and $\nu_q(\mathcal{N})\le s$, and set $\mathcal{H}:=\qbinom{W}{k}\setminus\mathcal{N}$. Then $|\mathcal{H}|=\qbinom{r}{k}_q-m_q(r,k,s)$. We first check that $e:=|\mathcal{H}|\ge2$, which is needed in order to apply Lemma~\ref{lem:induced-packing}. If $s=0$, then $\mathcal N=\varnothing$, and hence $e=\qbinom{r}{k}_q\ge2$, since $r\ge k+1$. Now assume $s\ge1$. Since $r=(s+1)k+(t-s-1)(k-1)\ge (s+1)k$, we may apply \eqref{eq:vector-space-emc-general} with $n=r$. Thus
$m_q(r,k,s)\le \qbinom{s}{1}_{q^k}\qbinom{r-1}{k-1}_q$. Writing $[a]_q:=\qbinom{a}{1}_q$, and using $\qbinom{r}{k}_q=([r]_q/[k]_q)\qbinom{r-1}{k-1}_q$ and $\qbinom{s}{1}_{q^k}=[sk]_q/[k]_q$, we obtain
\begin{equation*}
e
\ge
\left(\frac{[r]_q-[sk]_q}{[k]_q}\right)\qbinom{r-1}{k-1}_q
=
q^{sk}\frac{[r-sk]_q}{[k]_q}\qbinom{r-1}{k-1}_q.
\end{equation*}
Since $r-sk=k+(t-s-1)(k-1)\ge k$, we have $[r-sk]_q\ge [k]_q$. Therefore $e\ge q^{sk}\ge2$.

Apply Lemma~\ref{lem:induced-packing} to the family $\mathcal{H}$. We obtain pairwise disjoint copies $\mathcal{H}_1,\ldots,\mathcal{H}_m$ of $\mathcal{H}$, with $\mathcal{H}_i\subseteq\qbinom{W_i}{k}$ for some $r$-subspace $W_i\le V$, such that $m\ge (1-\gamma)\qbinom{n}{k}_q/e$, and conditions {\rm (i)} and {\rm (ii)} of Lemma~\ref{lem:induced-packing} hold. Since $e=\qbinom{r}{k}_q-m_q(r,k,s)$, it remains only to prove that $\{W_i\}_{i=1}^m$ is $t$-cover-free.

Suppose not. Then there are distinct indices $i_1,\ldots,i_t,j$ such that $W_j\le (W_{i_1}\cap W_j)+\cdots+(W_{i_t}\cap W_j)$. Relabeling, assume $j=t+1$ and $i_\ell=\ell$ for $1\le\ell\le t$. Put $A_i:=W_i\cap W_{t+1}$ for $i\in[t]$. By Lemma~\ref{lem:induced-packing}~{\rm (i)}, $\dim A_i\le k$, and since $W_{t+1}\le A_1+\cdots+A_t$, we have $W_{t+1}=A_1+\cdots+A_t$.

Set $B_0:=\{0\}$, $B_i:=A_1+\cdots+A_i$, and $d_i:=\dim B_i-\dim B_{i-1}$. Then $0\le d_i\le k$ for all $i$, and $\sum_{i=1}^t d_i=\dim W_{t+1}=r=(s+1)k+(t-s-1)(k-1)$. If fewer than $s+1$ of the $d_i$'s were equal to $k$, then $\sum_{i=1}^t d_i\le sk+(t-s)(k-1)=r-1$, a contradiction. Hence at least $s+1$ of the $d_i$'s are equal to $k$.

Whenever $d_i=k$, we have $\dim A_i=k$ and $A_i\cap B_{i-1}=\{0\}$. Thus the $A_i$'s corresponding to these indices are $k$-subspaces whose sum is direct. Relabeling these indices if necessary, assume that $A_1,\ldots,A_{s+1}$ are $k$-subspaces with direct sum. For each $i\in[s+1]$, we have $A_i=W_i\cap W_{t+1}$, so Lemma~\ref{lem:induced-packing}~{\rm (ii)} gives $A_i\notin\mathcal{H}_{t+1}$. Let $\mathcal{N}_{t+1}:=\qbinom{W_{t+1}}{k}\setminus\mathcal{H}_{t+1}$. Since $\mathcal{H}_{t+1}$ is a copy of $\mathcal{H}$, the family $\mathcal{N}_{t+1}$ is a copy of $\mathcal{N}$, and hence $\nu_q(\mathcal{N}_{t+1})=\nu_q(\mathcal{N})\le s$. However, $A_1,\ldots,A_{s+1}\in\mathcal{N}_{t+1}$ and their sum is direct, a contradiction.

Therefore $\{W_i\}_{i=1}^m$ is $t$-cover-free. This gives the desired lower bound. Since $\gamma>0$ was arbitrary, the proof is complete.
\end{proof}

\subsection{Proof of the key lemma}

Below we prove Lemma~\ref{lem:induced-packing}. We need the following standard Chernoff bound.

\begin{fact}[Chernoff bound]\label{fact:chernoff}
Let $X=X_1+\cdots+X_m$, where $X_1,\ldots,X_m$ are independent Bernoulli random variables, and let $\mu=\mathbb E X$. Then, for every $0<\eta<1$,
\begin{equation*}
\Pr[X\le (1-\eta)\mu]\le \exp(-\eta^2\mu/2)
\quad\text{and}\quad
\Pr[X\ge (1+\eta)\mu]\le \exp(-\eta^2\mu/3).
\end{equation*}
\end{fact}

For an $e$-uniform hypergraph $\mathcal{H}$ on $N$ vertices and two distinct vertices $x,y$, let $d_{\mathcal{H}}(x)$ denote the number of edges in $\mathcal{H}$ containing $x$, and let $d_{\mathcal{H}}(x,y)$ denote the number of edges in $\mathcal{H}$ containing both $x$ and $y$. The following lemma gives a sufficient condition under which a uniform hypergraph contains a near-perfect matching.

\begin{lemma}[{\cite{frankl1985near}}]\label{lem:Rodl-nibble}
For every integer $e\ge2$, every real $\mu\ge1$, and every $\varepsilon>0$, there exist $\delta=\delta(e,\mu,\varepsilon)>0$ and $D_0=D_0(e,\mu,\varepsilon)$ such that, for every $D\ge D_0$, the following holds. Every $e$-uniform hypergraph $\mathcal{H}$ on $N$ vertices satisfying
\begin{itemize}
    \item[{\rm (1)}] $d_{\mathcal{H}}(x)=(1\pm\delta)D$ for all but at most $\delta N$ vertices $x\in V(\mathcal{H})$;
    \item[{\rm (2)}] $d_{\mathcal{H}}(x)<\mu D$ for every $x\in V(\mathcal{H})$;
    \item[{\rm (3)}] $d_{\mathcal{H}}(x,y)<\delta D$ for every pair of distinct vertices $x,y\in V(\mathcal{H})$;
\end{itemize}
contains a matching of at least $(1-\varepsilon)N/e$ edges.
\end{lemma}

Finally, we present the proof of Lemma~\ref{lem:induced-packing}.

\begin{proof}[Proof of Lemma~\ref{lem:induced-packing}]
We may assume $0<\gamma<1$. Choose $0<\varepsilon<1$ such that $(1-\varepsilon)^2\ge1-\gamma$. Apply Lemma~\ref{lem:Rodl-nibble} with parameters $(e,\mu,\varepsilon)=(e,2,\varepsilon)$, and let $\delta$ and $D_0$ be the resulting constants. Decreasing $\delta$ if necessary, we may assume $0<\delta<1$. Choose $\rho>0$ so that $\sqrt\rho<\min\{\delta/4,\delta(1-\varepsilon)/2\}$.

By Lemma~\ref{lem:near-optimal-packing}, for all sufficiently large $n$, there exists an $[n,r,k+1]_q$-packing design $\mathcal P\subseteq\qbinom{V}{r}$ with $|\mathcal P|\ge (1-\rho)\qbinom{n}{k+1}_q/\qbinom{r}{k+1}_q$. Fix such $n$ and $\mathcal P$. Write $Q:=\qbinom{n}{k}_q$, $N:=\qbinom{r}{k}_q$, and $M:=\qbinom{n-k}{1}_q/\qbinom{r-k}{1}_q$. Set $h:=1/(10\max\{1,N-e\})$ and $p:=q^{-hn}$.

Let $\mathcal R\subseteq\qbinom{V}{k}$ be obtained by picking each $k$-subspace of $V$ independently with probability $1-p$. Define an $e$-uniform hypergraph $\mathcal J$ on vertex set $\mathcal R$ as follows: an $e$-set $\mathcal A\subseteq\mathcal R$ is an edge of $\mathcal J$ if $\mathcal A$ is a copy of $\mathcal F$ and $\qbinom{Y}{k}\cap\mathcal R=\mathcal A$ for some $Y\in\mathcal P$.

A matching in $\mathcal J$ gives the desired family. Indeed, if $\mathcal A_1,\ldots,\mathcal A_\ell$ is a matching in $\mathcal J$, then for each $i$ there exists $Y_i\in\mathcal P$ such that $\mathcal A_i=\qbinom{Y_i}{k}\cap\mathcal R$, and $\mathcal A_i$ is a copy of $\mathcal F$. Since $\mathcal P$ is an $[n,r,k+1]_q$-packing design, distinct $Y_i,Y_j$ satisfy $\dim(Y_i\cap Y_j)\le k$. Moreover, if $A=Y_i\cap Y_j$ has dimension $k$, then $A\notin\mathcal A_i$ and $A\notin\mathcal A_j$, since otherwise $A\in\mathcal A_i\cap\mathcal A_j$, contradicting that the $\mathcal A_i$'s are disjoint. Thus it remains to show that $\mathcal J$ has a matching of size at least $(1-\gamma)Q/e$.

For $A\in\qbinom{V}{k}$, let $d_{\mathcal P}(A):=|\{Y\in\mathcal P:A\le Y\}|$. We first note that $d_{\mathcal P}(A)\le M$ for every $A$. Indeed, every $Y\in\mathcal P$ containing $A$ contains exactly $\qbinom{r-k}{1}_q$ $(k+1)$-subspaces through $A$, and no such $(k+1)$-subspace can lie in two different members of $\mathcal P$. Since the total number of $(k+1)$-subspaces of $V$ through $A$ is $\qbinom{n-k}{1}_q$, we get $d_{\mathcal P}(A)\le M$.

By double counting pairs $(A,Y)$ with $A\le Y$ and $Y\in\mathcal P$, we have
\begin{equation*}
\frac{1}{Q}\sum_{A\in\qbinom{V}{k}}d_{\mathcal P}(A)
=
\frac{|\mathcal P|\qbinom{r}{k}_q}{\qbinom{n}{k}_q}
\ge (1-\rho)M.
\end{equation*}
Since $d_{\mathcal P}(A)\le M$ for each $A\in\qbinom{V}{k}$, all but at most $\sqrt\rho Q$ of the $k$-subspaces $A$ satisfy $d_{\mathcal P}(A)\ge(1-\sqrt\rho)M$. Let $\mathcal G:=\{A\in\qbinom{V}{k}:(1-\sqrt\rho)M\le d_{\mathcal P}(A)\le M\}$. Then $|\mathcal G|\ge(1-\sqrt\rho)Q$.

Let $\lambda$ be the number of distinct copies of $\mathcal F$ in
$\qbinom{W}{k}$ that contain a fixed $k$-subspace of $W$. This is well defined because ${\rm GL}(W)$ acts transitively on $\qbinom{W}{k}$, and $\lambda\ge1$. Set $D:=M\lambda(1-p)^{e-1}p^{N-e}$. Since $M=\Theta(q^{n-r})$, $p\to0$, and $p^{N-e}=q^{-hn(N-e)}\ge q^{-n/10}$, we have $D\to\infty$. 

Fix $A\in\mathcal G$. For each $Y\in\mathcal P$ with $A\le Y$, let $X_{A,Y}$ be the indicator of the event that $\qbinom{Y}{k}\cap\mathcal R$ is a copy of $\mathcal F$, conditioned on $A\in\mathcal R$. There are exactly $\lambda$ possible $e$-subsets producing such a copy, and each occurs with probability $(1-p)^{e-1}p^{N-e}$. Hence $\Pr[X_{A,Y}=1]=\lambda(1-p)^{e-1}p^{N-e}$.

If $Y_1,\ldots,Y_\ell\in\mathcal P$ all contain $A$, then the random variables $X_{A,Y_1},\ldots,X_{A,Y_\ell}$ are independent. Indeed, for $i\ne j$, we have $\qbinom{Y_i}{k}\cap\qbinom{Y_j}{k}=\{A\}$; otherwise, a second $k$-subspace $B\ne A$ contained in both $Y_i$ and $Y_j$ would make $A+B$ contain a $(k+1)$-subspace lying in two members of $\mathcal P$, contradicting the packing property.

Therefore, for every $A\in\mathcal G\cap\mathcal R$, the degree $d_{\mathcal J}(A)$ is a sum of independent Bernoulli variables with mean $\mu_A:=\mathbb E[d_{\mathcal J}(A)\mid A\in\mathcal R]=d_{\mathcal P}(A)\lambda(1-p)^{e-1}p^{N-e}$. Hence $(1-\sqrt\rho)D\le\mu_A\le D$. Moreover, for arbitrary $A\in\qbinom{V}{k}$, we still have $\mathbb E[d_{\mathcal J}(A)\mid A\in\mathcal R]\le D$.

If $A,B\in\mathcal R$ are distinct, then at most one member of $\mathcal P$ contains both $A$ and $B$, since otherwise $A+B$ would contain a $(k+1)$-subspace lying in two members of $\mathcal P$. Thus $d_{\mathcal J}(A,B)\le1$.

We next show that, for all sufficiently large $n$, there exists a realization of $\mathcal R$ satisfying:
\begin{itemize}
    \item[{\rm (a)}] $|\mathcal R|\ge(1-\varepsilon)Q$;
    \item[{\rm (b)}] $d_{\mathcal J}(A)=(1\pm\delta)D$ for every $A\in\mathcal G\cap\mathcal R$;
    \item[{\rm (c)}] $d_{\mathcal J}(A)<2D$ for every $A\in\mathcal R$.
\end{itemize}

For (a), since $p\to0$, we may assume $1-p\ge1-\varepsilon/2$. Then $\mathbb E|\mathcal R|=(1-p)Q\ge(1-\varepsilon/2)Q$, and Chernoff's bound gives $\Pr[|\mathcal R|<(1-\varepsilon)Q]=o(1)$.

For (b), fix $A\in\mathcal G$ and condition on $A\in\mathcal R$. Since $\sqrt\rho<\delta/4$, we have $\mu_A\ge(1-\delta/4)D$, and hence $(1-\delta)D\le(1-\delta/2)\mu_A$. Chernoff's bound gives $\Pr[d_{\mathcal J}(A)\notin(1\pm\delta)D\mid A\in\mathcal R]\le2e^{-c_\delta D}$ for some constant $c_\delta>0$. Since $Q$ is exponential in $n$, while $D=q^{\Omega(n)}$, a union bound over $A\in\mathcal G$ gives
\begin{equation*}
\Pr[\exists A\in\mathcal G\cap\mathcal R:~d_{\mathcal J}(A)\notin(1\pm\delta)D]=o(1).
\end{equation*}

For (c), for every fixed $A$, the conditional random variable $d_{\mathcal J}(A)\mid(A\in\mathcal R)$ is a sum of independent Bernoulli variables with mean at most $D$. Chernoff's bound again gives $\Pr[d_{\mathcal J}(A)\ge2D\mid A\in\mathcal R]\le e^{-D/3}$, and another union bound gives $\Pr[\exists A\in\mathcal R:~d_{\mathcal J}(A)\ge2D]=o(1)$.

Thus there is a realization of $\mathcal R$ satisfying (a)--(c). Fix such a realization. Since $|\qbinom{V}{k}\setminus\mathcal G|\le\sqrt\rho Q$, property (a) gives $|\mathcal R\setminus\mathcal G|\le \sqrt\rho Q\le \sqrt\rho|\mathcal R|/(1-\varepsilon)<\delta|\mathcal R|$. Hence all but at most $\delta|\mathcal R|$ vertices of $\mathcal J$ have degree $(1\pm\delta)D$. Moreover, by (c), every vertex has degree $<2D$, and, as shown above, every pair of distinct vertices has codegree $\le 1$.

Applying Lemma~\ref{lem:Rodl-nibble} to $\mathcal J$, we obtain a matching of size at least $(1-\varepsilon)|\mathcal R|/e$. By (a), this is at least $(1-\varepsilon)^2Q/e\ge(1-\gamma)Q/e$. As explained at the beginning, this matching yields pairwise disjoint copies $\mathcal F_1,\ldots,\mathcal F_m$ of $\mathcal F$ satisfying {\rm (i)} and {\rm (ii)}, with $m\ge(1-\gamma)\qbinom{n}{k}_q/e$. This completes the proof.
\end{proof}

\section*{Acknowledgements}
The authors are grateful to Bo Ning for bringing \cite{ning2020formula} to their attention, to Andrey Kupavskii for pointing out \cite{kolupaev2023erdHos}, and to Ferdinand Ihringer for sending them \cite{ihringer2026structure}. They also thank Jian Wang and Benjian Lv for several helpful discussions. The research is supported by the National Natural Science Foundation of China under Grant Nos. 12571352 and 12231014, and the Fundamental Research Funds for the Central Universities.

\bibliographystyle{plain}
\normalem
\bibliography{ref}

\appendix

\section{Parameter calculations for Theorem~\ref{thm:vemc}~{\rm (iii)}}\label{app:1}

We write $[a]_q:=\qbinom{a}{1}_q$. The following elementary estimate will be used repeatedly.

\begin{lemma}\label{lem:B-estimate}
Let $m,r,t$ be positive integers with $m\ge t+r$. Then
\begin{equation*}
\bigl(1-4q^{-(m-t-r+2)}\bigr)[t]_q\qbinom{m-1}{r-1}_q
\le B_q(m,r,t)\le [t]_q\qbinom{m-1}{r-1}_q .
\end{equation*}
\end{lemma}

\begin{proof}
By Lemma~\ref{lem:counting}, we have
\begin{equation*}
B_q(m,r,t)=\sum_{i=0}^{t-1}q^{ir}\qbinom{m-i-1}{r-1}_q .
\end{equation*}
For $0\le i\le t-1$, the product formula gives
\begin{equation*}
\frac{q^{ir}\qbinom{m-i-1}{r-1}_q}{\qbinom{m-1}{r-1}_q}
=
q^i\prod_{j=0}^{r-2}
\frac{1-q^{-(m-i-1-j)}}{1-q^{-(m-1-j)}} .
\end{equation*}
Each factor in the product is at most $1$, and hence the upper bound follows by summing $q^i$ over $0\le i\le t-1$.

For the lower bound, set $x_{i,j}:=q^{-(m-i-1-j)}$ and $y_j:=q^{-(m-1-j)}$. Then $0<y_j\le x_{i,j}<1$, and
\begin{equation*}
\frac{1-x_{i,j}}{1-y_j}\ge 1-2x_{i,j}.
\end{equation*}
Thus the product is at least $1-2\sum_{j=0}^{r-2}x_{i,j}$. Since $i\le t-1$, we have
\begin{equation*}
\sum_{j=0}^{r-2}x_{i,j}
\le q^{-(m-t-r+2)}(1+q^{-1}+\cdots+q^{-(r-2)})
\le 2q^{-(m-t-r+2)}.
\end{equation*}
Therefore the product is at least $1-4q^{-(m-t-r+2)}$. Summing over $i$ gives the desired lower bound.
\end{proof}

\begin{proposition}\label{prop:B-parameter}
Let $n,k,s$ be positive integers with $s\ge2$ and $k\ge2$, and let $q$ be a prime power. If $n\ge 2sk-s+k+1$, then
\begin{equation*}
\frac{B_q(n,k,s)}{[sk]_q}
>
B_q(n-1,k-1,sk).
\end{equation*}
\end{proposition}

\begin{proof}
By Lemma~\ref{lem:B-estimate}, $B_q(n-1,k-1,sk)\le [sk]_q\qbinom{n-2}{k-2}_q$. Hence it is enough to prove
\begin{equation*}
B_q(n,k,s)>[sk]_q^2\qbinom{n-2}{k-2}_q .
\end{equation*}
Applying Lemma~\ref{lem:B-estimate} to $B_q(n,k,s)$, we get
\begin{equation*}
B_q(n,k,s)\ge
\bigl(1-4q^{-(n-s-k+2)}\bigr)[s]_q\qbinom{n-1}{k-1}_q .
\end{equation*}
Since $\qbinom{n-1}{k-1}_q=([n-1]_q/[k-1]_q)\qbinom{n-2}{k-2}_q$, it remains to show
\begin{equation*}
\bigl(1-4q^{-(n-s-k+2)}\bigr)[s]_q\frac{[n-1]_q}{[k-1]_q}>[sk]_q^2 .
\end{equation*}
We have $[n-1]_q/[k-1]_q>q^{n-k}$, $[sk]_q<q^{sk}$, and, since $q^s\ge4$,
\begin{equation*}
\frac{[sk]_q}{[s]_q}
=
\frac{q^{sk}-1}{q^s-1}
<
\frac{q^{sk}}{q^s-1}
\le \frac43 q^{s(k-1)}.
\end{equation*}
Thus $[sk]_q^2/[s]_q<(4/3)q^{2sk-s}$. On the other hand, $n\ge2sk-s+k+1$ gives $q^{n-k}\ge q^{2sk-s+1}$, and $n-s-k+2\ge2s(k-1)+3\ge7$. Hence $1-4q^{-(n-s-k+2)}>3/4$. Therefore
\begin{equation*}
\bigl(1-4q^{-(n-s-k+2)}\bigr)\frac{[n-1]_q}{[k-1]_q}
>
\frac{3}{4}q^{2sk-s+1}\geq\frac{3}{2}q^{2sk-s}
>
\frac43 q^{2sk-s}
>
\frac{[sk]_q^2}{[s]_q}.
\end{equation*}
This proves the proposition.
\end{proof}

\section{Parameter calculations for Theorem~\ref{thm:q-HM}}\label{app:2}

We again use Lemma~\ref{lem:B-estimate}.

\begin{proposition}\label{prop:H-parameter}
Let $n,k,s$ be positive integers with $s\ge2$ and $k\ge2$, and let $q$ be a prime power. If $n\ge 2sk-s+k+3$, then
\begin{equation*}
\frac{H_q(n,k,s)}{[sk]_q}
>
B_q(n-1,k-1,sk).
\end{equation*}
\end{proposition}

\begin{proof}
By Lemma~\ref{lem:B-estimate}, it is at most $[sk]_q\qbinom{n-2}{k-2}_q$. Hence it is enough to prove
\begin{equation*}
H_q(n,k,s)>[sk]_q^2\qbinom{n-2}{k-2}_q .
\end{equation*}
Since $H_q(n,k,s)\ge B_q(n,k,s-1)$, Lemma~\ref{lem:B-estimate} gives
\begin{equation*}
H_q(n,k,s)\ge
\bigl(1-4q^{-(n-s-k+3)}\bigr)[s-1]_q\qbinom{n-1}{k-1}_q .
\end{equation*}
Using $\qbinom{n-1}{k-1}_q=([n-1]_q/[k-1]_q)\qbinom{n-2}{k-2}_q$, it remains to show
\begin{equation*}
\bigl(1-4q^{-(n-s-k+3)}\bigr)[s-1]_q\frac{[n-1]_q}{[k-1]_q}>[sk]_q^2 .
\end{equation*}
We have $[n-1]_q/[k-1]_q>q^{n-k}$, $[sk]_q<q^{sk}$, and $[s-1]_q\ge q^{s-2}$. Hence $[sk]_q^2/[s-1]_q<q^{2sk-s+2}$. Meanwhile, $n\ge2sk-s+k+3$ gives $q^{n-k}\ge q^{2sk-s+3}$, and $n-s-k+3\ge2s(k-1)+6\ge10$. Thus $1-4q^{-(n-s-k+3)}>1/2$. Consequently
\begin{equation*}
\bigl(1-4q^{-(n-s-k+3)}\bigr)\frac{[n-1]_q}{[k-1]_q}
>
\frac{1}{2}q^{2sk-s+3}
\ge
q^{2sk-s+2}
>
\frac{[sk]_q^2}{[s-1]_q}.
\end{equation*}
This proves the proposition.
\end{proof}

\end{document}